\documentclass[11pt,a4paper]{amsart}
\usepackage{graphicx}
\usepackage{amssymb}
\usepackage[all]{xy}
\usepackage{amsmath}
\usepackage[french,english]{babel}
\usepackage{amssymb,amsmath,amsfonts,amsthm,amscd}
\usepackage{indentfirst,graphicx}
\usepackage[font={scriptsize},captionskip=5pt]{subfig}

\def\Projan{\mathop{\rm Projan}}
\newtheorem{theorem}{Theorem}[section]
\newtheorem{exa}[theorem]{Example}

\newtheorem{rem}[theorem]{Remark}
\newenvironment{remark}{\begin{rem} \em}{\end{rem}}
\newtheorem{cor}[theorem]{Corollary}
\newenvironment{corollary}{\begin{cor} \em}{\end{cor}}
\newtheorem{tont}[theorem]{Definition}
\newenvironment{definition}{\begin{tont} \em}{\end{tont}}
\newtheorem{lem}[theorem]{Lemma}

\newtheorem{Prop}[theorem]{Proposition}
 
\let\:=\colon

\begin{document}

\title[Collections of 1-forms and
Chern Numbers]{The Multiplicity Polar Theorem, Collections of
1-forms and Chern Numbers}
\author{TERENCE GAFFNEY and NIVALDO G.  GRULHA Jr}

\maketitle

\selectlanguage{english}

\begin{abstract}
In this work we show how the Multiplicity Polar Theorem can be used
to calculate Chern numbers for collection of 1-forms.
\end{abstract}

\section{Introduction}

Given a space with singularities, and a geometric invariant defined
for smooth spaces, it is interesting to see whether or not the
invariant is well defined for the singular space, and, if it is,
what are the contributions to the invariant from the singularities.

In a series of papers Ebeling and Gusein-Zade have discussed the
meaning of such invariants as, the index of a differential form
\cite{E-GZ3}, various notions of the  index of a vector field
\cite{EG}
   and the Chern numbers, and have described the contributions from the singularities in some cases. In \cite {E-GZ3},
they calculated the radial index of a 1-form on a complete
intersection singularity. In \cite{Gaff-AG}, it is shown that in the
case of a differential 1-form with an isolated singularity on
$X\subset\mathbb{C}^n$, $X$, a complex analytic space, that the
radial index can be computed using the multiplicity of a pair of
modules. The computation of \cite{Gaff-AG} amounts to computing the
intersection multiplicity of the graph of the one form $\omega$,
which is a subset of the (unprojectivised) conormal bundle of
$\mathbb{C}^n$ and the cotangent space of $X$. In contrast to \cite
{E-GZ3}, the calculation is valid for any equidimensional space.

In \cite {EG2} Ebeling and Gusein-Zade developed the notion of the
Chern number of a singular space using collections of differential 1
forms. Their numbers are well defined for any equidimensional
reduced complex analytic germ, but they only compute the number for
ICIS singularities. Their Chern number is again an intersection
number. As in the earlier work, the intersection takes place at the
level of conormal spaces; they call the points in $X$ which are the
projection of the points of intersection, {\it special points}.

In the case that we have just one 1-form, the Chern number
is the Euler obstruction of the differential form (\cite {EG} p17).  This is related to  the Euler obstruction of a set and  the Euler
obstruction of a function as defined by Brasselet, Massey,
Parameswaran and Seade in \cite{BMPS}. In \cite{EG05a} the
definition of the Euler obstruction of a function was adapted to the
case of $1$-form, the Euler obstruction of a function was studied by
several authors, in this direction we have for example the papers
\cite{grulha1},\cite{STV},\cite{BG},\cite{Gaff-AG}. In \cite{BGR} the authors determine relations between the local Euler obstruction of an analytic map $f$ defined in \cite{grulha2} and the Chern number of a
convenient collection of 1-forms associated to $f$.

In this work we use the multiplicity polar theorem to calculate
Chern numbers for any equidimensional reduced complex analytic germ.
This extends the earlier work of \cite{Gaff-AG}. For the Chern
number problem one must work with a set of collections of
differential one forms, and calculate the order of the point where
all of the collections are linearly dependent. Since we want to
calculate the number of points at which the fields are linearly
dependent after a generic perturbation, this is again a problem
involving modules. Because we have a collection of forms we have a
collection of modules, so the problem is like a problem in
intersection theory, except the spaces are defined by modules not by
ideals.

The computation of the Chern numbers is an example of a problem where the underlying vector bundle, which is the tangent bundle in the Chern case, is not defined at every point of $X$. The set $X$ must be modified; so we pass to the Nash modification of $X$ where the tangent bundle of $X$ is defined in order to understand the problem fully.  This process of modifying a space to fill in points where a bundle $\xi$ fails to be defined works in general, and our process of calculating intersection numbers  also extends. We outline this in the last section. This suggests the easier problem of calculating Chern numbers when the tangent bundle is well defined at all points, and more generally, intersection numbers of modules. We take this up in section two. The main themes of section four appear in this material.

Also in section two, we recall some basic ideas about the theory of
integral closure of modules and the statement of the multiplicity
polar theorem.

In section three we recall how Ebeling and Gusein-Zade  develop the
notion of Chern number in their paper.

In section 4 we introduce the  notion of a special point for a collection. Roughly speaking  a point $p \in X$ is called a special point of
the collection $\{\omega_{j}^{(i)}\}$ of $1$-forms on the variety
$X$ if there exists a point in the fiber of the Nash modification over $p$ such that  the restriction of the
$1$-forms $\omega_{1}^{(i)}, \cdots, \omega_{d-k_{i}+1}^{(i)}$ to
the point are linearly dependent
for each $i=1, \cdots, s$. We next see how special points can be viewed as intersections, hence have an associated intersection number, if isolated.
 We then begin to solve the ``module intersection
theoretic problem"  for the computation of the Chern numbers described above.  We prove a "Gysin" type
theorem, (Proposition 4.7) that is, under suitable genericity
hypotheses, we can do our calculations on a single space which
represents the intersection of all but the last spaces defined by
our collection and use the last module associated with the
collection restricted to this space for our computations. We also
prove a genericity result (Proposition 4.10) which shows that by
deforming just the last collection of differential forms, we can
ensure the set of collections is generic in an appropriate sense.

 We begin section 5 by recalling a result of  Ebeling and Gusein-Zade (Proposition 5.1) relating Chern numbers and special points. In Proposition 5.4 we describe in integral closure terms what it means for $x\in X$ not to be a special point for a collection of forms. After gearing up to apply the multiplicity polar theorem, in Theorem 5.13
we show that deforming our last collection allows us to split the contribution of the Chern number from an isolated special point into the multiplicity of a pair of modules and the intersection number of the new collection. Using this as the inductive step, we can write the contribution to the Chern number as a sum of multiplicities of pairs (Corollary 5.14).

We next show that if $X$ is an ICIS, then our formula agrees with that of Ebeling and Gusein-Zade (Corollary 5.15 and the discussion afterwards.)

We close by indicating how our results can be generalized to the case of a  bundle $E^k$ defined on a Zariski open, everywhere dense  subset $U$ of an analytic space $X$, $E^k$ a sub-bundle with $k$ dimensional fiber of a bundle, $F^l$, where $F^l$ is defined everywhere.

The authors thank Steven Kleiman for helpful conversations on the connection between their work and the intersection multiplicity of Serre.

\vspace{2cm}

\section{Integral closure of modules}

Let  $(X, x)$ be a germ of a complex analytic space and $X$ a
small representative of the germ and let $\mathcal{O}_{X}$ denote the
structure sheaf on a complex analytic space $X$. The study of what
it means for a collection of $1$-forms to have a special point on a
singular space depends on the behavior of limiting tangent
hyperplanes. The key tool for studying these limits is the theory of
integral closure of modules, which we now introduce.

\begin{definition} Suppose $(X, x)$ is the germ of a complex analytic space,
$M$ a submodule of $\mathcal{O}_{X,x}^{p}$. Then $h \in
\mathcal{O}_{X,x}^{p}$ is in the integral closure of $M$, denoted
$\overline{M}$, if for all analytic $\phi : (\mathbb{C}, 0) \to (X,
x)$, $h \circ \phi \in (\phi^{*}M)\mathcal{O}_{1}$. If $M$ is a
submodule of $N$ and $\overline{M} = \overline{N}$ we say that $M$
is a reduction of $N$.
\end{definition}

To check the definition it suffices to check along a finite number of curves whose generic point is in the Zariski open subset of $X$ along which $M$ has maximal rank. (Cf. \cite {G-2})

If a module $M$ has finite colength in $\mathcal{O}_{X,x}^{p}$, it
is possible to attach a number to the module, its Buchsbaum-Rim
multiplicity. We can also define the multiplicity of a pair of
modules $M \subset N$, $M$ of finite colength in $N$, as well, even
if $N$ does not have finite colength in $\mathcal{O}_{X}^{p}$.

In studying the geometry of singular spaces, it is natural to study
pairs of modules. In dealing with non-isolated singularities, the
modules that describe the geometry have non-finite colength, so
their multiplicity is not defined. Instead, it is possible to define
a decreasing sequence of modules, each with finite colength inside
its predecessor, when restricted to a suitable complementary plane.
Each pair controls the geometry in a particular codimension.

We recall how to construct the multiplicity using the approach of
Kleiman and Thorup \cite{KT}. Given a submodule $M$ of a free
$\mathcal{O}_{X^{d}}$ module $F$ of rank $p$, we can associate a
subalgebra $\mathcal{R}(M)$ of the symmetric $\mathcal{O}_{X^{d}}$
algebra on $p$ generators. This is known as the Rees algebra of $M$.
If $(m_{1}, \cdots ,m_{p})$ is an element of $M$ then $\sum
m_{i}T_{i}$ is the corresponding element of $\mathcal{R}(M)$. Then
$\Projan(\mathcal{R}(M))$, the projective analytic spectrum of
$\mathcal{R}(M)$ is the closure of the projectivised row spaces of
$M$ at points where the rank of a matrix of generators of $M$ is
maximal. Denote the projection to $X^{d}$ by $c$. If $M$ is a
submodule of $N$ or $h$ is a section of $N$, then $h$ and $M$
generate ideals on $\Projan \mathcal{R}(N)$; denote them by $\rho(h)$
and $\rho(\mathcal{M})$. If we can express $h$ in terms of a set of
generators $\{n_{i}\}$ of $N$ as $\sum g_{i}n_{i},$ then in the
chart in which $T_{1}\neq 0,$ we can express a generator of
$\rho(h)$ by $\sum g_{i}T_{i}/T_{1}.$ Having defined the ideal sheaf
$\rho(\mathcal{M}),$ we blow up by it.

On the blowup $B_{\rho(\mathcal{M})}(\Projan \mathcal{R}(N))$ we have
two tautological bundles. One is the pullback of the bundle on
$\Projan \mathcal{R}(N)$. The other comes from $\Projan
\mathcal{R}(M)$. Denote the corresponding Chern classes by $c_{M}$
and $c_{N}$, and denote the exceptional divisor by $D_{M,N}$.
Suppose the generic rank of $N$ (and hence of $M$) is $g$.

Then the multiplicity of a pair of modules $M, N$ is:

$$
e(M,N) = \sum_{j=0}^{d+g-2}\int D_{M,N}\cdot c_{M}^{d+g-2-j}\cdot
c_{N}^{j}.
$$

Kleiman and Thorup show that this multiplicity is well defined at $x
\in X$ as long as $\overline{M} = \overline{N}$ on a deleted
neighborhood of $x$. This condition implies that $D_{M,N}$ lies in
the fiber over $x$, hence is compact.

We will be interested in computing the change in the multiplicity of
the pair $(M,N),$ denoted $\Delta(e(M,N))$. We will assume that the
integral closures of $M$ and $N$ agree off a set $C$ of dimension
$k$ which is finite over $Y$, and assume we are working on a
sufficiently small neighborhood of the origin, so that every component
of $C$ contains the origin in its closure. Then $e(M,N, y)$ is the
sum of the multiplicities of the pair at all points in the fiber of
$C$ over $y$, and $\Delta(e(M,N))$ is the change in this number from
$0$ to a generic value of $y.$ If we have a set $S$ which is finite
over $Y$, then we can project $S$ to $Y$, and the degree of the
branched cover at $0$ is $mult_{y} S.$ (Of course, this is just the
number of points in the fiber of $S$ over our generic $y.$)

We also need the notion of the polar varieties of $M$. The {\it polar variety of codimension $k$} of $M$ in $X$, denoted
$\Gamma_k(M)$, is constructed by intersecting $\Projan{\mathcal R}(M)$
 with $X\times H_{g+k-1}$ where
$H_{g+k-1}$ is a general plane of codimension $g+k-1$, then projecting to
$X$.

Setup: We suppose we have families  of modules $M\subset  N$, $M$ and $N$
 submodules of a free module $F$ of rank $p$
 on an equidimensional family of spaces with equidimensional
 fibers ${\mathcal X}^{d+k}$, ${\mathcal X}$ a family over a smooth base
$Y^k$. We assume that the generic rank of $M$, $N$ is $e\le p$.  Let
$P(M)$ denote $\Projan {\mathcal R}(M)$, $\pi_M$
 the projection to ${\mathcal X}$.

Let $C(M)$ denote the locus of points where $M$ is not free, ie. the
points where the rank of $M$ is less
than $g$, $C(\Projan {\mathcal R}(M))$
its inverse image under $\pi_M$.

We can now state the Multiplicity Polar Theorem.

\begin{theorem}(Multiplicity Polar Theorem) Suppose in the above
setup we have that $\overline{M} = \overline{N}$ off a set $C$ of
dimension $k$ which is finite over $Y$. Suppose further that
$C(Projan\mathcal{R}(M))(0) = C(Projan\mathcal{R}(M(0)))$ except
possibly at the points which project to $0 \in \mathcal{X}(0).$
Then, for y a generic point of $Y$, $$\Delta(e(M,N)) =
mult_{y}\Gamma_{d}(M) - mult_{y}\Gamma_{d}(N).$$
\end{theorem}

Now we show how this machinery can be applied to a module intersection problem. Suppose we are given modules $M_1\subset F_1$ and $M_2\subset F_2$, $F_i$ free $ {\mathcal{O}}_{X^{d},x}$ modules  of rank $p_i$, $M_i$ generated by $n_i$ generators. Suppose the codimension of $C(M_i)$ is $n_i-p_i+1$, and the sum of the codimensions is $d$, $C(M_i)$ equidimensional. If we deform the generators of $M_i$, how many points do we expect to see where both modules have less than maximal rank?

We can take this number as the intersection number of the two modules.

As further justification, we relate this number to an intersection number at $x$. Let $\mathcal{M}(p,q)$, $p\leq q$, be the space of $p\times q$
matrices with complex entries and let $D_{p,q}$ be the subspace of
$\mathcal{M}(p,q)$ consisting of matrices of rank less than $p$. The
subset $D_{p,q}$ is an irreducible subvariety of $\mathcal{M}(p,q)$
of codimension $q-p+1$.

Fix a matrix of generators $[M_i]$ of $M_i$.

Then each matrix $[M_i]$ defines a section $\Gamma_{M_i}$ of  $ \mathbb{C}^{n}\times\mathcal{M}(p_i,n_i)$ in the obvious way; the pair
defines a section $\Gamma_{M_1,M_2}$ of $ \mathbb{C}^{n}\times\mathcal{M}(p_1,n_1)\times \mathcal{M}(p_2,n_2)$.
 We will assume that
$X\times D_{p_1,n_1}\times D_{p_2,n_2} \cap Im(\Gamma_{M_1,M_2})$ is
isolated and lies over $x$. The intersection number of  $X\times D_{p_1,n_1}\times D_{p_2,n_2}$ and $Im(\Gamma_{M_1,M_2})$  at $(x, \Gamma_{M_1,M_2}(x))$ is the number we want to calculate.

\begin{theorem} Suppose each of the sections $\Gamma_{M_i}$ is transverse to $ X^d\times D_{p_i,q_i}$ on a Z-open set $U_i$ such that $U_i\cap C(M_i)$ is Z-open and dense in $C(M_i)$. Then the intersection number of  $X\times D_{p_1,n_1}\times D_{p_2,n_2}$ and $Im(\Gamma_{M_1,M_2})$  at $(x, \Gamma_{M_1,M_2}(x))$ is $e(M_1, {\mathcal{O}}_{C(M_2),x})=e(M_2,  {\mathcal{O}}_{C(M_1),x})$.
\end{theorem}

\begin{proof} $X$ may be singular, so we assume $X$ is stratified with the canonical Whitney stratification. Then the transversality of $\Gamma_{M_i}$ means transversality to each $S_j\times D_{p_{i},q_{i}}$, $S_j$ a stratum. This ensures the generic point of each component of $C(M_i)$ is a smooth point of $X$. It also ensures that the codimension of $ {\mathcal{O}}_{C(M_i),x}$ is $n_i-p_i+1$. Since $X\times D_{p_1,n_1}\times D_{p_2,n_2} \cap Im(\Gamma_{M_1,M_2})$ is
isolated and lies over $x$, the sum of the codimensions of the $C(M_i)$ is $d$.

 Let us show that the intersection number of  $X\times D_{p_1,n_1}\times D_{p_2,n_2}$ and $Im(\Gamma_{M_1,M_2})$  at $(x, \Gamma_{M_1,M_2}(x))$ is $e(M_1,  {\mathcal{O}}_{C(M_2),x})$. The proof of the other half of the inequality is parallel.

 Note that the number of generators of $M_1$ as a $ {\mathcal{O}}_{C(M_2),x}$  module is
 $$n_1=(n_1-p_1+1)+p_1-1$$
 $$=d-(n_2-p_2+1)+p_1-1$$
 $$={\rm dim}\hskip 2pt C(M_2) + p_1 -1.$$

Then  by Theorem 1.2 of \cite{G-7}, we can find a perturbation of $[M_1]$ by a matrix of generic constants such that the section induced by the new matrix, $[\widetilde M_1]$, of $C(M_2)\times \mathcal{M}(p_1,n_1)$
  is transverse to  $C(M_2)\times D_{p,q}$, and the finite number of points at which $[\widetilde M_1]$ has less than maximal rank  occur at smooth points of $C(M_2)$ and there are $e(M_1, \mathcal{O}_{C(M_2),x})$ of them. In particular, $x$ is no longer a point where both sections have less than maximal rank. (It is not hard to see from the proof of Theorem 1.2 that in fact these lie in the Z-open dense subset of $C(M_2)$ on which the section $\Gamma_{M_2}$ is transverse
 to $X\times D_{p_2,n_2}$.) The transversality conditions on  $\Gamma_{M_2}$ and  $\Gamma_{\widetilde M_1}$ imply that 
 the section $\Gamma_{\widetilde M_1,M_2}$ is transverse to $X\times D_{p_1,n_1}\times D_{p_2,n_2}$ at all points of intersection. The total number of such points counted with multiplicity is the intersection number of  $X\times D_{p_1,n_1}\times D_{p_2,n_2}$ and $Im(\Gamma_{M_1,M_2})$  at $(x, \Gamma_{M_1,M_2}(x))$; the transversality statement implies each point occurs with multiplicity 1.

\end{proof}

\begin{corollary} Suppose $\mathcal{O}_{X,x}$ is Cohen-Macaulay, then the intersection number of  $X\times D_{p_1,n_1}\times D_{p_2,n_2}$ and $Im(\Gamma_{M_1,M_2})$  at $(x, \Gamma_{M_1,M_2}(x))$ is the colength of the ideal generated by the  maximal minors of $[M_i]$, $i=1,2$.
\end{corollary}

\begin{proof} Since $\mathcal{O}_{X,x}$ is Cohen-Macaulay and the structure on  $\mathcal{O}_{C(M_2),x}$ given by the minors of $[M_2]$ is generically reduced, it is reduced and $ \mathcal{O}_{C(M_2),x}$ is Cohen-Macaulay. Then $e(M_1, \mathcal{O}_{C(M_2),x})$ is the colength of the ideal of minors of $[M_1]$ in $ \mathcal{O}_{C(M_2),x}$ which gives the result.\end{proof}

Looking at the proof of the above theorem, in applying the technique of proof to geometric problems, we see that we need a description of the desired quantity as an intersection number, and a theorem about the transversality of a deformation of $[M_1]$ by a matrix of generic constants.

If $\xi_1$ and $\xi_2$ are vector bundles, we may wish to calculate geometric invariants related to sections of the bundles.  If the desired invariant is supported at a point, then locally the set of sections of our vector bundles are free modules, and we can look at the submodules generated by the given sets of sections. Then the last theorem can be used to calculate the contribution to the invariant at a point where the sections fail to be generic.

In the next couple of sections we will look at a more difficult case, one in which the vector bundle may only be defined on a Z-open subset of $X$. This will involve modifying $X$ to produce a new space on which the bundle is defined, then taking into account the fiber of the modification over $x$.

Before developing these ideas, we mention the connection between the ideas of this section and the intersection multiplicity defined by Serre (\cite{S}). Given modules $M_1\subset F_1$ and $M_2\subset F_2$, $F_i$ free $ {\mathcal{O}}_{X^{d},x}$ modules  of rank $p_i$ as above, Serre's intersection number is the alternating sum of the lengths of the $Tor^i(F^{p_1}/M_1, F^{p_2}/M_2)$. Under the hypotheses of Theorem 2.3, Serre's intersection number is the same as $e(M_1, {\mathcal{O}}_{C(M_2),x})=e(M_2,  {\mathcal{O}}_{C(M_1),x})$. This holds because under small deformations of the $M_i$ the intersection number doesn't change; but then, by a small deformation, we can reduce to the ideal case (ie. the modules have rank one less than maximum at common points where they have less than maximal rank). Then the intersection number counts the same points as $e(M_1, {\mathcal{O}}_{C(M_2),x})$.  For the case where  $\mathcal{O}_{X,x}$ is Cohen-Macaulay more can be said. Here the complex used to compute the $Tor^i$ is exact, so Serre's intersection number is just the length of $F^{p_1}/M_1\otimes F^{p_2}/M_2$. To see this,  consider the complex for ${O_n}^{p_1}/M_1$. At points where $M_1$ has maximal rank, this complex is exact. Further all the maps have maximal rank. Now tensor with ${O_n}^{p_2}/M_2$, and consider the resulting complex. At points where $M_2$ has maximal rank we are tensoring with 0, so the complex is exact. At points where $M_2$ has less than maximal rank distinct from the origin, then the complex remains exact, since all the maps have maximal rank. So the origin is the only point where the complex is not exact; but by the acyclicity lemma, (cf \cite {E} p498)  the complex must be exact there as  well.

We obtain the following result, which extends some theorems of Buchsbaum and Rim  (\cite {B-R} 2.4 p.207, 4.3 and
4.5 p.223), as a corollary:

\begin{corollary} Suppose $\mathcal{O}_{X,x}$ is Cohen-Macaulay, then the length of $F^{p_1}/M_1\otimes F^{p_2}/M_2$ is the colength of the ideal generated by the  maximal minors of $[M_i]$, $i=1,2$.
\end{corollary}

\section{Indices of collections of 1-forms}

W. Ebeling and S. M. Gusein-Zade studied indices for collections of
1-forms \cite{EG, EG2}, in this section we will recall some ideas
and notation from their papers about these concepts.

If $P$ is a complex analytic manifold of dimension $n$, then its
Euler characteristic $\chi(P)$ is the characteristic number
$$\langle c_{n}(TP), [P]\rangle = (-1)^{n}\langle c_{n}(T^{*}P),
[P]\rangle,$$ where $TP$ is the tangent bundle of the manifold $P$,
$T^{*}P$ is the dual bundle, and $c_{n}$ is the corresponding Chern
class and $[P]$ the fundamental class of $P$.

The top Chern class of a vector bundle is the first obstruction to
the existence of a non-vanishing section. Other Chern classes are
obstructions to the existence of a linearly independent collection of
sections. There, instead of $1$-forms on a complex variety, we
consider collections of $1$-forms. Further, to calculate intersections of Chern Classes and hence Chern numbers, we will need collections of collections of $1$-forms.

Let $\pi\: E \to P$ be a complex analytic vector bundle of rank m
over a complex analytic manifold $P$ of dimension $n$. It is known
that the  ($2(n-k)$-dimensional) cycle Poincar\'{e} dual to the
characteristic classe $c_{k}(E)$  $(k = 1, \cdots, m)$ is
represented by the set of points of the manifold $P$ where $m-k+1$
generic sections of the vector bundle $E$ are linearly dependent.

We continue to use the notation of section two: Let $\mathcal{M}(p,q)$, $p\leq q$, be the space of $p\times q$
matrices with complex entries and let $D_{p,q}$ be the subspace of
$\mathcal{M}(p,q)$ consisting of matrices of rank less than $p$. The
subset $D_{p,q}$ is an irreducible subvariety of $\mathcal{M}(p,q)$
of codimension $q-p+1$. The complement $W_{p,q} =
\mathcal{M}(p,q)\setminus D_{p,q}$ is the Stiefel manifold of
$p$-frames in $\mathbb{C}^{q}$ . It is known that the Stiefel
manifold $W_{{p,q}}$ is $2(q-p)$-connected and
$H_{2(q-p)+1}(W_{p,q})\cong \mathbb{Z}$.

We now develop the notation necessary to handle collections of
collections of forms. For the rest of the paper, we  will  refer to
these objects simply as collections.

Let $\textbf{k}=(k_{1},\cdots, k_{s})$ be a sequence of positive
integers with $\sum_{i=1}^{s}k_{i}=k$. Consider the space
$\mathcal{M}_{m,\textbf{k}}= \prod_{i=1}^{s}
\mathcal{M}(m-k_{i}+1,m)$ and the subvariety
$D_{m,\textbf{k}}=\prod_{i={1}}^{s}D_{m-k_{i}+1,m}$ in it. The
variety $D_{m,\textbf{k}}$ consists of sets $\{A_{i}\}$ of
$(m-k_{i}+1\times m)$ matrices such that ${\rm  rk} \hskip 2pt A_{i}
< m-k_{i}+1$ for each $i=1, \cdots, s.$ Since $D_{m,\textbf{k}}$ is
irreducible of codimension $k$, its complement $W_{m,\textbf{k}}=
\mathcal{M}_{m,\textbf{k}}\setminus D_{m,\textbf{k}}$ is
$(2k-2)$-connected, $H_{2k-1}(W_{m,\textbf{k}})\cong \mathbb{Z}$,
and there is a natural choice of a map from an oriented manifold of
dimension $2k-1$ to the manifold $W_{m,\textbf{k}}$.

Let $(X^{d},0)\subset (\mathbb{C}^{n},0)$ be the germ of a purely
$n$-dimensional reduced complex analytic variety at the origin. For
$\textbf{k}=\{k_{i}\}, i=1, \cdots, s, j=1, \cdots, d-k_{i}+1$, let
$\{\omega_{j}^{(i)}\}$ be a collection of germs of $1$-forms on
$(\mathbb{C}^{n},0)$. (Note that $\{\omega_{j}^{(i)}\}$ for a fixed value of $i$, is itself a collection of $d-k_{i}+1$  $1$-forms.) Let $\varepsilon > 0$ be small enough so that
there is a representative $X$ of the germ $(X,0)$ and
representatives $\{\omega_{j}^{(i)}\}$ of the germs of $1$-forms
inside the ball $B_{\varepsilon}(0) \subset \mathbb{C}^{n}$.

The kind of points whose multiplicity we wish to compute is
described in the next section.

\section{Special Points}

\begin{definition} A point $p \in X$ is called a special point of
the collection $\{\omega_{j}^{(i)}\}$ of $1$-forms on the variety
$X$ if there exists a sequence $p_{m}$ of points from the
non-singular part $X_{reg}$ of the variety $X$ such that the
sequence $T_{p_{m}}X_{reg}$ of the tangent spaces at the points
$p_{m}$ has a limit $L$ (in $G(d,n)$) and the restriction of the
$1$-forms $\omega_{1}^{(i)}, \cdots, \omega_{d-k_{i}+1}^{(i)}$ to
the subspace $L \subset T_{p}\mathbb{C}^{n}$ are linearly dependent
for each $i=1, \cdots, s$. The collection $\{\omega_{j}^{(i)}\}$ of
$1$-forms has an isolated special point on $(X,0)$ if it has no
special point on $X$ in a punctured neighborhood of the origin.

\end{definition}
Notice that we require each element in the collection to be linearly
dependent when restricted to the same limit plane. Notice also, that
if an element of the collection has less than maximal rank at a
point, then it is linearly dependent on all planes passing through
the point.

The framework of this section is a variation on the setting used in
\cite{EG}. In developing the properties of special points, it is helpful to work on two levels, one of which is based on the Nash modification. The Nash modification comes into play because the tangent bundle of $X$ is not defined at singular points of $X$. However the Nash bundle is an extension of the tangent bundle on the modified space. We begin to describe this setting.

Let $\{\omega^{(i)}_{j}\}$ be a collection of germs of 1-forms on
$(X, 0)$ with an isolated special point at the origin. Let $\nu:
\widetilde{X} \to X$ be the Nash transformation of the variety $X$,
and $\widetilde{T}$ the Nash bundle. The collection of 1-forms
$\{\omega^{(i)}_{j}\}$ gives rise to a section $\Gamma({\omega})$ of
the bundle

$$\widetilde{\mathbb{T}} =
\bigoplus_{i=1}^{s}\bigoplus_{j=1}^{n-k_{i}+1}
\widetilde{T}^{*}_{i,j}$$
where $\widetilde{T}^{*}_{i,j}$ are copies of the dual Nash bundle
$\widetilde{T}^{*}$ over the Nash transform $\widetilde{X}$ numbered
by indices $i$ and $j$.

Let $\mathbb{D} \subset \widetilde{\mathbb{T}}$ be the set of pairs
$(x, \{\alpha^{(i)}_{j}\})$ where $x \in \widetilde{X}$ and the
collection $\{\alpha^{(i)}_{j}\}$ is such that $\alpha^{(i)}_{1}
,\cdots, \alpha^{(i)}_{n-k_{i}+1}$ are linearly dependent for each
$i = 1,\cdots., s.$

\begin{definition} The local Chern obstruction $Ch_{X,0}\{ \omega^{(i)}_{j}\}$ of the
collections of germs of 1-forms $\{ \omega^{(i)}_{j}\}$ on $(X,0)$
at the origin is the obstruction to extend the section
$\Gamma({\omega})$ of the fibre bundle
$\widetilde{\mathbb{T}}\setminus \mathbb{D} \to \widetilde{X}$ from
the preimage of a neighbourhood of the sphere $S_{\varepsilon} =
\partial B_{\varepsilon}$ to  $\widetilde{X}$, more precisely its value, as an element of the cohomology group $H^{2n}(\nu^{-1}(X \cap
B_{\varepsilon}), \nu^{-1}(X \cap S_{\varepsilon}), \mathbb{Z})$, on
the fundamental class of the pair $(\nu^{-1}(X \cap
B_{\varepsilon}), \nu^{-1}(X \cap S_{\varepsilon})).$
\end{definition}

The computation of the local Chern obstruction will be revisited in section 5.

The other setting for the study of special points is closer to $X$, and we describe it next. This setting will allow us to describe the number of special points as an intersection number.

Let $X^{d} \subset\mathbb{C}^{n}$, $\mathcal{L}^{k}$ be the set of
collections of $1$-forms respecting the partition of k as above $(k=k_{1}+k_{2}+ \cdots + k_{s})$,  
$\mathbb{D}^{k}_{X} \subset \mathbb{C}^{n}\times\mathcal{L}^{k}$ be
the closure of the set of pairs $(x, \{l^{i}_{j}\})$ such that $x
\in X_{reg}$ and the restriction of the linear functions $l^{i}_{1},
\cdots, l^{i}_{n-k_{i}+1}$ to $T_{x}X_{reg} \subset \mathbb{C}^{N}$
are linearly dependent for each $i=1, \cdots, s$.

Notice that the fiber of $\mathbb{D}^{k_i}_{X} $ over a regular
point $x$ of $X$ can be identified with the elements of $M(d-k_i+1,n)$
which have less than maximal rank when restricted to $T_{x}X$. Since
$T_{x}X$ is defined by $n-d$ equations,  the fiber of
$\mathbb{D}^{k_i}_{X} $ is itself a  fibration over the singular
matrices in $M(d-k_i+1,d)$, hence the restriction of
$\mathbb{D}^{k_i}_{X} $ to the regular points of $X$ has a
stratification by rank of the collection restricted to $T X_{x}$, $ x \in X_{reg}$.

The collection of $1$-forms $\{\omega^{(i)}_{j} \}$ defines a
section $\Gamma_{\omega}$ of $ \mathbb{C}^{n}\times\mathcal{L}^{k}$;
 we will assume in our results that the image of the
projection $\pi_{X}(\mathbb{D}_{X}^{k}\cap Im(\Gamma_{\omega}))$ is
isolated. Note that this implies that the intersection of  $Im(\Gamma_{\omega})$ and $\mathbb{D}_{X}^{k}$ is isolated as well,  since
$\pi_{X}\vert_{Im(\Gamma_{\omega})}$ is $1-1$. We will further assume that sets  $Im(\Gamma_{\omega})$ and $\mathbb{D}_{X}^{k}$ have complementary dimension viewed as subsets of $ \mathbb{C}^{n}\times\mathcal{L}^{k}$. Thus, their intersection number is well defined.

We are interested in computing this  intersection number.

As we will see, this amounts to computing the intersection number of a collection of sets defined by modules. The viewpoint of this paper is to compute this intersection number by successively restricting to the intersection of  $k-1$ elements of the collection. There is a technical condition which describes the way a ``good"  collection of these sets meet, given in Definition 4.5, which needs some preparation.
\begin{definition} Given  a pair $(x,P)$, $x\in X$,  $P$ in $G(d,n)$ the pair is degenerate for the collection $\{\omega_{j}\}$ $1\leq j\leq d-k+1$ at $x$ if
$\{\omega_{j}\}\vert_{P}$ is linearly dependent at $x$. Denote the
set of degenerate pairs for $\{\omega\}$ by $\mathbb{B}(\omega)$.
\end{definition}

\begin{Prop} Suppose the collection
$\{\omega\}$ is linearly independent at the origin. Then
$\mathbb{B}(\omega)$ has codimension $k$ in $X\times G(d,n)$.

\end{Prop}

\begin{proof} We can cover $G(d,n)$ with open sets as follows: pick
a coordinate plane $P$ of dimension $d$ and a plane of complementary dimension using the complementary coordinates, which we denote by $\hat P$. Clearly, the complementary plane intersects $P$ only at the origin. Consider all planes which are the graphs of a linear
map from $P$ to  $\hat P$. The equations of these graphs
give a unique set of equations describing the plane, and thus associate a matrix of size $(n-d\times n)$ to each plane.

These planes are just the planes that intersect $\hat P$ at the origin, and thus are a Zariski open subset of $G(d,n)$.

Suppose $U$ is such an open set, then construct the map from $U
\times Hom(n,d-k+1)$ to $ Hom(n, n-k+1)$ by combining the $2$ matrices - the element of
$Hom(n,d-k+1)$ and the matrix of equations describing points of $U$.

This matrix has size $(n-k+1 \times n)$ and as a map from $$U\times
Hom(n, d-k+1) \to Hom(n, n-k+1)$$ is transverse to the rank
stratification. So the codimension of the set of pairs which give
matrices of less than maximal rank is $(n)-(n-k+1)+1=k$.

Working globally, it is clear that the set of degenerate pairs is a fibration over the set of elements of $Hom(n,d-k+1)$ of maximal rank. So fixing
$\omega$ we get that the set of degenerate planes has codimension
$k$. (Also fixing a plane $P$, the set of $L \in Hom(n,d-k+1)$ for
which the plane is degenerate also has codimension k.)
\end{proof}

If we have a collection of forms $\{\omega^{(i)}_{j} \}$ with
$\sum_{i=1}^{s}k_{i}=k$, every element of which is linearly
independent at the origin, then $\mathbb{B}(\omega)$ denotes the
planes which are degenerate for every element of the collection. It
has codimension less than or equal to $k$.

\begin{definition} Given $X^{d},0 \subset \mathbb{C}^{n},0$ with $0\in S(X)$ and a
collection  $\{\omega^{(i)}_{j} \}$with  $\sum_{i=1}^{s}k_{i}=k$,
$k\le d$ such that each element of the collection is linearly
independent at 0, we say that the collection is proper for $X^d$ if
$dim(\widetilde{X}(S(X))\cap \mathbb{B}(\omega)) \leq d-k-1$ where
$\widetilde{X}(S(X))$ is the restriction  of the Nash modification
of $X$ to $S(X)$, the singular set of $X$. If this condition holds
for a  collection of forms  linearly
independent at 0, with the exception of components
of the intersection over the origin, we say the collection is proper
on a deleted neighborhood of the origin.

If $X$ is smooth at $0$, then we ask $dim(\widetilde{X}(0)\cap \mathbb{B}(\omega)) \leq d-k-1$

\end{definition}

\begin{remark}The dimension $\widetilde{X}(S(X))$ is at most $d-1$; if there
is a component of dimension $d-1$, the condition just asks that the
component meets $B(\omega)$ properly. Since the dimension of all components of $\widetilde{X}\cap B(\omega)$ is at least $d-k$, the properness condition implies that a point of $\widetilde{X}(S(X))\cap \mathbb{B}(\omega)$ is in the closure of points of the intersection lying over smooth points of $X$. Note also that if $k=d$, and the collection is proper, then $\widetilde{X}(S(X))\cap \mathbb{B}(\omega)$ is empty.
\end{remark}

For the geometric description we need of special points, it is
necessary to lift our constructions to the level of the Nash
modification.

On $\widetilde{X}\times \mathcal{L}^{k}$ we can consider triples
$(x, P, L)$ where $P$ is a degenerate plane for $L$. Call the space of
triples $\mathbb{D}^{k}$. It is clearly a fibration over
$\widetilde{X}$.

Thinking of  $\mathbb{C}^{n}\times G(d,n)\times\mathcal{L}^{k}$ as a trivial fibration over $\mathbb{C}^{n}\times G(d,n)$, we have the
section induced by $\omega$ which we  denote by $\Gamma_{\omega,
G}$. Note that if we restrict  $\Gamma_{\omega,
G}$ to $\widetilde X$, then $\Gamma_{\omega,
G}^{-1}(\mathbb{D}^{k})=\mathbb{B}(\omega)\cap \widetilde X$.  Now the image of $\Gamma_{\omega,G}$ has dimension $n+d(n-d)$,
while $\mathbb{D}^{k}$ has codimension $(n-d)+d(n-d)+k$ so the
expected dimension of the intersection is $d-k$. Denote the
projection of the intersection to $X$ by $S(\omega)$. We can make
$k$ a multi-index and make similar constructions; we get the
expected dimension of $S(\omega^{(i)}_{j})$ is $d-(k_{1}+ \cdots +
k_{s})$.

Suppose $\{\omega^{(i)}_{j}\}$ is a collection of $1$ forms such
that the $\sum k_{i}=d$ and $0$ is an isolated special point. Then
all of the various $S(\omega^{(i)}_{j})$ using different
subcollections must have the correct expected dimension; for if
$S(\omega^{(i)}_{j})$ is too large for one subcollection, the excess
dimension will be passed to the others and $0$ will not be isolated.

Denote $\mathbb{D}^{k}\cap Im(\Gamma_{\omega,G})$ by
$S_{N}(\omega)$.

We will also be interested in the notion of a restricted special point; given a collection of $1$-forms  $\omega_{1}^{(i)}, \cdots, \omega_{d-k_{i}+1}^{(i)}$, $1\le i\le s$, we say $p$ is a {\it restricted special point} of the collection if it is a special point, and the sequence of points  $p_{m}$ are in $S(\omega^{(i)}_{j})$, $1\le i\le  s-1$. In the next proposition we will prove that if the collection $\omega^{(i)}_{j})$, $1\le i\le  s-1$ is proper, then every special point is a restricted special point. \vspace{.3cm}

\subsection{Setup} Here we describe our assumption about the
collections.

Let $X^{d},0 \subset \mathbb{C}^{n},0$ and $\{\omega^{(i)}_{j}\}$, a
collection of $1$-forms with $1\leq i \leq s$, $1\leq j \leq
d-k_{i}+1$, $\sum k_{i}=d.$

Assume the collection arranged so that the first $r$ collections are
 $1$-forms which are linearly independent at $0$. We assume
the $1$-forms in the remaining collection are all linearly dependent
at the origin. We assume the collection has an isolated singularity
at the origin, and that the generic point of $S(\omega^{(i)}_{j})$,
$1\leq i \leq s-1$, is in $X_{reg}$. If $r=s$ we also assume that the
collection made up of the first $s-1$ elements is proper for $X$.

\begin{Prop} If, in the above set-up $0$ is a isolated special point of
the collection then there exists a curve $C$ on
$S(\omega^{(i)}_{j})$, $1\leq i \leq s-1$, generically in $X_{reg}$
such that $\{\omega\}$ is linearly dependent when restricted to the
limiting tangent plane $T$ at the origin, and the origin is the only
point on $S(\omega^{(i)}_{j})$ with this property.
\end{Prop}
\begin{proof} There are two cases to consider.

$1-$ Assume $r < s$, assume a special point exists. This is also a
special point for the collection with the first $s-1$ elements. Thus
$S(\omega^{(i)}_{j})$, $1\leq i \leq s-1$, has positive dimension and
its generic point is in $X_{reg}$. For $C$, use any curve on
$S(\omega^{(i)}_{j})$, and let $T$ be the limit tangent plane,
$T_{t}$ the tangent plane to $X$ at point $t$ on curve $C$.

Then $\{\omega^{(i)}_{j}\}$, $1\leq i \leq s-1$, are linearly
dependent on $T_{t}$ since our point is in $S(\omega^{s}_{j})$ hence
they are linearly dependent on $T$.

Since $\{\omega^{(s)}_{j}\}$ is linearly dependent at zero,  they
are linearly dependent on $T$ also.

Clearly, any point for which such a curve exists is a special point,
so the origin is the only such point.

$2-$ $(r=s) -$ Assume we have a special point then
$S(\omega^{(i)}_{j})$, $i\leq s-1$, has positive dimension with
generic point in $X_{reg}$. Denote the limit tangent plane on which
all of our collections restrict to be linearly dependent by $T$.

By the properness assumption, no component of $S_{N}(\omega)$,
$i\leq s-1$ can lie over $S(X)$; for every component of
$S_{N}(\omega)$ must have dimension $d-k$, $k=\sum k_i$, $1\le i\le
s-1$, while by the properness assumption the points over $S(X)$ must
have dimension $d-1-k$ or less.

This implies there exist a curve $\varphi: \mathbb{C},0 \to X,0$
generic point in $X_{reg}\cap S(\omega^{i}_{j})$ with $i\leq s-1$
such that  the limit tangent plane is $T$.

Now all the members of our collection are linearly dependent on $T$
including $\{\omega^{(s)}_{j}\}$.
\end{proof}

The previous proposition explains why we are interested in collections which are proper. The properness condition means that if we have a special point, then it is a restricted special point as we can realize the limiting plane on which the collection is dependent as a limit of tangents to $X$ along a curve in some $S(\omega^{(i)}_{j})$. This is the key to our ability to study the intersection number of $Im(\Gamma_{\omega})$ and $\mathbb{D}_{X}^{k}$ by restricting to $S(\omega^{(i)}_{j})$.

There is a converse to the proposition which requires a stronger
genericity condition.

\begin{Prop} Suppose in the setup of this section, the collection made
up of the first $s-1$ elements is proper for $X$, if the elements of the
collection are linearly independent at the origin, and if not, they are
proper for $X$ on a deleted neighborhood of the origin.
Suppose the origin is the only point where there exists a curve $C$
on $S(\omega^{(i)}_{j})$, $1\leq i \leq s-1$ generically in $X_{reg}$
such that $\{\omega\}$ is linearly dependent when restricted to the
limiting tangent plane $T$ at the origin. Then the origin is an
isolated special point of the collection.\end{Prop}

\begin{proof} Clearly the origin is a special point. If it were non-isolated,
then we could apply the previous proposition to find curves detecting the
nearby special points as well. \end{proof}

These two propositions show that when studying the behavior of
special points, with the right genericity requirements, we can
restrict from $X$ to $S(\omega^{(i)}_{j})$, $1\leq i \leq s-1$, and
having done so consider only the last element of the collection.

The next proposition serves as a ``moving lemma".
% by varying only
%the last element of the collection we can put the collection in
%``general position".

This proposition and its corollary, together with the
multiplicity polar theorem, will show that the invariant of the next
section  computes the intersection number of $Im(\Gamma_{\omega})$
and $\mathbb{D}_{X}^{k}$. The argument we give is adapted from that
appearing in \cite{G-7}, Theorem 1.2, p187.

To prove our proposition we want to consider the map:

$$\Theta: \mathbb{D}_{X}^{k_s}\vert_{S(\omega^{(i)})_{i\leq
s-1}} \times M(d-k_{s}+1,n) \to M(d-k_{s}+1,n)$$ where $\Theta
((x,L),M) \to L - (\omega(x)+M)$.

 If we resolve the singularities of the set
$ \mathbb{D}_{X}^{k_s}\vert_{S(\omega)}$, then the composition $\Theta \circ \pi
_{ \mathbb{D}_{X}^{k_s}}$ is a submersion because of the  contribution from the
$M$ term.

In resolving these singularities there may be multiple components. For example, for $X^2$ a surface in $\mathbb{C}^3$ with an isolated singularity at the origin, then if $s=2$ and $\omega$ consists of two forms then $ \mathbb{D}_{X}^{k_2}\vert_{S(\omega)}$
has $(0,M(2,3))$ as a component. This follows because the polar curve of $X^2$ is non-empty, which implies that the generic element of $M(2,3)$ has less than maximal rank on a curve on $X^2$, hence lies in $ \mathbb{D}_{X}^{k_2}$ along that curve.  However, there will be a unique component for each component of ${S(\omega)}$ which surjects onto that component. Denote the components of $ \mathbb{D}_{X}^{k_s}\vert_{S(\omega)}$ which surject onto ${S(\omega)}$ by $ \mathbb{D}_{S(\omega)}^{k_s}$. The fiber of these components over the origin are those collections of forms which are the limits of forms degenerate along a curve in
${S(\omega)}$.

Let $C$ denote  $\Theta^{-1}(0)\cap ( \mathbb{D}_{S(\omega)}^{k_s}\times M(d-k_{s}+1,n))$ and consider the projection $p$ from $C
\to M(d-k_{s}+1,n)$.

Now $$\dim C = \dim S(\omega^{(i)})_{i\leq s-1} +\dim (\text{generic fiber of }D^{k_{s}})=k_{s} + [(d-k_{s}+1)(n)-k_{s}].$$

The dimension of $S(\omega^{(i)})_{i\leq s-1}$ is $k_s$, because, since we have an isolated singularity, the dimension of $S(\omega^{(i)})_{i\leq s}$ must be $0$.

The expression in $[ \hspace{.5cm} ]$ above holds because the
codimension of $D^{k_{s}}$ in $M(,)$ is just
$$-[(d-k_{s}+1)+(n-d)]+n+1=k_{s},$$ so the map $C \to M(d-k_{s}+1,n)$
is a map between equidimensional spaces. Assuming that $0\in X$ is an isolated special point of the collection, the fiber over $0$ of $p$ is
a single point $\omega(0)$.

Earlier in this section we began to look at the intersection number of $Im(\Gamma_{\omega})$ and $\mathbb{D}_{X}^{k}$. Restricting to $S(\omega^{(i)})_{i\leq s-1}$  we can also look at the intersection number of $Im(\Gamma_{\omega^{(s)}})$ and $ \mathbb{D}_{S(\omega)}^{k_s}$. Our moving lemma will be used to calculate this piece of the intersection number of $Im(\Gamma_{\omega})$ and $\mathbb{D}_{X}^{k}$.

\begin{definition} A special point of a collection $\{\omega^{(i)}_{j}\}$
of germs of $1$-forms on $X$ is non-degenerate, if the section
$\Gamma_{\omega^{(i)}_{j}} {1\leq i\leq s}$, meets
$\mathbb{D}_{X}^{k}$ transversally at the point.
\end{definition}

We can now state our proposition.

\begin{Prop}\label{proposicao} Given a collection as in the set up of this section, assume that the section
$\Gamma_{\omega^{(i)}_{j}} $, ${1\leq i\leq s-1} $ meets
$\mathbb{D}_{X}^{k}$ transversely on a Z-open subset of
${S(\omega^{(i)})_{i\leq s-1}}$. Then for generic $M$, the
collection $\{ \omega^s+tM\}$ meets
$\mathbb{D}_{S(\omega)}^{k_s}$ transversely at all points close to the origin
for $t$ sufficiently small, $t\ne 0$. The number of such points
is just the degree of the projection from $C$ to $M(d-k_{s}+1,n)$
over the origin in $M(d-k_{s}+1,n)$. Further, each such point is a non-degenerate point of the collection $\{\omega^{(i)}_{j}\}$, ${1\leq i\leq s} $.\end{Prop}
\begin{proof}

Pick $M$ in the complement of the $\Delta(p)$, the discriminant of
the projection from $C$ to $M(d-k_{s}+1,n)$, such that the line
between $0$ and $M$ does not intersect these sets close to $0$.

Over the points of this line close to $0$, the number of points is
the degree of $p$ and $p$ is a submersion at each point. This
implies that the map obtained by fixing the $M$ term in $\Theta$  is
	a submersion also. Note that the dimension of the source and
target of this map are the same, hence the map is in fact a
diffeomorphism.

We are interested in exploring the consequences of this fact.

Let us first consider the case where at the points on the fiber of
$p$ over $tM$, $t$ small, $x$ is in the regular part of
${S(\omega^{(i)})_{i\leq s-1}}$, and the element in
$\mathbb{D}^{k_{s}}$ has rank one less than maximal when restricted
to $TX_x$. Then the resolution of $D^{k_{s}}\vert_{S(\omega)}$ is an
equivalence at such points because $D^{k_{s}}\vert_{S(\omega)}$ is
smooth there, so we can work on the tangent space of
$D^{k_{s}}\vert_{S(\omega)}$. At each point this splits into a direct
sum--the part along ${S(\omega^{(i)})_{i\leq s-1}}$, and the part
along the fiber. There is a similar decomposition of the tangent
space of the target--the part which can be identified with the fiber
in the source, and the normal space to this. The differential is the
identity on the tangent space to the fiber, so since the
differential is surjective, the restriction of the differential to
the tangent space to ${S(\omega^{(i)})_{i\leq s-1}}$ must surject
onto the normal space to the fiber. In turn this implies that the
section induced from $\omega +tM$ meets  $D^{k}$ transversely. In
fact, since for transversality we just need the tangent vectors to
$S(\omega^{(i)})_{i\leq s-1}$, and  the other elements of the
collection intersect $\mathbb{D}_{X}^{k}$  transversely, the
collection  $\{(\omega^{(i)})_{i\leq s-1}, \omega^s+tM\}$ meets
$\mathbb{D}_{X}^{k}$ transversely.

In the general case, note that the assumptions we made above
coincide with the resolution being an equivalence. If the resolution
is not an equivalence, then some tangent vectors on the resolution
will be in the kernel of the differential of the projection, hence
the differential will not be surjective, contradicting our choice of
$M$. So we only need to consider the above special case.

\end{proof}
\begin{remark} If $\{ \omega^{(i)}\}$, $1\le i\le r$ is the maximal subset of our collections which meet properly at the origin, then we can choose $M$ so that $\{ \omega^{(i)}, \omega^s+tM\}$, $1\le i\le r$ also meet properly at the origin for $t\ne 0$. This will be implicit in our application of our moving lemma.
\end{remark}

Denote the collection obtained by moving our last element by ${\widetilde{\omega}}$.
\begin{cor} In the set-up of last proposition we have

$$Im(\Gamma_{\omega})\cdot \mathbb{D}_{X}^{k}=
\Gamma(\omega_s)\cdot \mathbb{D}^{k_{s}}|_{S(\omega^{(i)})_{i\leq s-1}}+ Im(\Gamma_{\widetilde{\omega}})\cdot \mathbb{D}_{X}^{k}.$$ \end{cor}

\begin{proof} The effect of moving $\{\omega^{(s)}\} $ is to split off points from the intersection $Im(\Gamma_{\omega})\cdot \mathbb{D}_{X}^{k}$. The first intersection number on the right is the degree of the projection from $C$ to $M(d-k_{s}+1,n)$, and   this is the number of points split off from the intersection number on the left hand side of the equation. Moving $\{\omega^{(s)}\}$  ensures that the intersection  $\Gamma(\omega_s+tM)\cdot \mathbb{D}^{k_{s}}|_{S(\omega^{(i)})_{i\leq s-1}}$ is void at the origin, ie the intersection point at the origin has split into non-degenerate points. The second term on the right is the remaining points at the origin.
\end{proof}
\begin{cor} In the set-up of last proposition, suppose in addition that the collection $\{ \omega^{(i)}\}$, $1\le i\le s-1$ is proper. Then
$$Im(\Gamma_{\omega})\cdot \mathbb{D}_{X}^{k}=
\Gamma(\omega_s)\cdot \mathbb{D}^{k_{s}}|_{S(\omega^{(i)})_{i\leq s-1}}.$$
\end{cor}

\begin{proof} Since the collection  $\{ \omega^{(i)},\widetilde{\omega^{(s)}}\}$ , $1\le i\le s-1$ is proper, the intersection of
$\Gamma_{\{ \omega^{(i)},\widetilde{\omega^{(s)}}\}}$ and $ \mathbb{D}_{X}^{k}$ is empty.

\end{proof}

\begin{remark} If $X^{n-1}$ is a hypersurface and $\omega_{i}$ a collection of forms
with $(n-1)-k+1$ elements, which are linearly independent at the
origin, then it is easy to check if $\widetilde{X}(0)\cap B(\omega)$
has dimension  $(n-1)-k-1$.

Suppose $\dim \widetilde{X}(0)\cap B(\omega)\geq (n-1)-k$. To each
point in $B(\omega)$ there corresponds a unique point in $Proj(\omega)$,
the projectivized row space of $\omega_{i}$.

Note that points of $Proj(\omega)$ corresponding to points of
$\widetilde{X}(0)\cap B(\omega)$ are limiting tangent hyperplanes to
$X$ at the origin, so the set of points of $Proj(\omega)$ which are
limiting tangent hyperplanes has dimension $\geq (n-1)-k = dim
Proj(\omega)$ so every point is a limiting tangent hyperplane.

This is true if and only if $JM(f, \sum \alpha_{i} \omega_{i})$
fails to be a reduction of $JM(f)\oplus\mathcal{O}_{X}$ for all
$\alpha_{i}$. This can be checked using curves.
\end{remark}

\begin{remark}We continue with the hypersurface isolated singularity case.

Suppose $j+1$ collections $\{\omega^{i}\}$ $1\le i\le j+1$ are in
general position i.e. all are linearly independent at $0$ and ${\rm
dim} \cap Proj(\{\omega^{i}\})$ is $(n-1)-\sum_i k_{i}$. Suppose the
properness condition holds for the first $j$ elements but fails for
the collection. A dimension count shows that a whole component of
$\cap Proj\{\omega^{i}\}$ must lie in the fiber of the Nash
modification over the origin.
 Again
this is easy to check.
\end{remark}

\section{Computing Chern Numbers}

In this section we will apply the Multiplicity Polar Theorem to
relate the number of non-degenerate special points of a collection
of 1-forms with the Chern number of this collection. Ebeling and
Gusein-Zade proved this next proposition.

\begin{Prop}\cite{EG} Let $X$ be a representative of the germ of a
complex analytic space, then the local Chern obstruction
$Ch_{X,0}\{ \omega^{(i)}_{j}\}$ of a collection
$\{ \omega^{(i)}_{j}\}$ of germs of holomorphic 1-forms is equal to
the number of special points on X of a generic  deformation of the
collection.
\end{Prop}

If $X$ is defined by $F : \mathbb{C}^{n} \to \mathbb{C}^{p},$ then
the Jacobian module of $X$ denoted $JM(X),$ is the submodule of
$\mathcal{O}^{p}_{X}$ generated by the partial derivatives of $F.$
Given a collection of $1$-forms with $r$ elements defined on $X$,
form the $p+ r$ by $n$ matrix $D(F, \omega)$ by augmenting the
Jacobian matrix $DF$ at the bottom with the $1$-forms from the
collection. Call the submodule of the free module
$\mathcal{O}^{p+r}_{X},$ generated by the columns of $\left(
  \begin{array}{c}
    D(F) \\
    \omega \\
  \end{array}
\right)$, the augmented Jacobian module and denote it by $JM(X,
\omega).$

Note that this construction works in general. Given a submodule $M$ of a free module $F$, one can select a matrix of generators, and augment the matrix using linear forms. The points at which the new matrix has less than maximal rank is independent of the choice of generators of $M$ as the row space doesn't change.

In the next lemma, we begin to relate the theory of integral closure and the infinitesimal limiting geometry of our sets of forms.

\begin{lem} Let $X$ be a representative of the germ of a
complex analytic space, and let $\mathcal{L}= \{l_{1},
l_{2},\cdots,l_{s}\}$ be a collection of linear forms.  Consider the hyperplanes
defined by the  forms $\sum a_{i}l_{i}$. None of these hyperplanes is a limiting
tangent hyperplane to $X,0$ at the origin if and only if
$\overline{JM(X)_{p}}= \overline{JM(X)}$, where p is a submersion
whose kernel is the intersection of the kernels of $l_{1}, \cdots,l_{s}$.

\noindent (Here $JM(X)_{p}$ is the submodule of $JM(X)$ generated by
$\frac{\partial}{\partial v_{i}}f$ where the $v_{i}$ span the  kernel of $p$.)

\end{lem}

\begin{proof}

Let us prove this result in the special case when $p$ is a linear
projection on the last $s$ variables.

If $JM(X)_{p}$ is a reduction of $JM(X)$, then so is $JM(X)_{h}$,
where $h= \sum a_{i}l_{i}$ because $ker(h) \supset ker(p)$. Hence, the hyperplane defined by $h$ is not a limiting tangent hyperplane.

Let us prove now that if $\overline{JM(X)_{h}}= \overline{JM(X)}$
for all $h$ then $\overline{JM(X)_{p}}= \overline{JM(X)}$.

Let $K= Ker(p)$, we will show $JM(X)_{p} \subset \mathcal{O}_{X}^{k}$ is a
reduction of $JM(X) \subset \mathcal{O}_{X}^{k}$ if every hyperplane
that contains $K$ is not a limiting tangent hyperplane.

Suppose $JM(X)_{p}$ is not a reduction, this implies that exist
$\varphi: \mathbb{C} \to X,0$ and non-zero $l:\mathbb{C}^{n}\to
\mathbb{C}$, such that if N is matrix of generators of $JM(X)_{p}$ and $M$ is the
matrix of generators of $JM(X)$, then the ideal generated by the components of
$(l\cdot N)\circ \varphi(t)$  has larger order than the ideal generated by the components of $l \cdot M \circ \varphi$. Denote the order of $l \cdot M \circ \varphi$ by $k$. Then

$$\lim 1/t^{k} <m_{1}(t), \cdots, m_{n}(t)> $$ defines a limiting tangent hyperplane. Since
$m_{1}=\eta_{1}, \cdots, m_{p}=\eta_{p}$, and the order of these terms is greater than $k$, it follows that  T is a limiting tangent hyperplane which contains the
kernel of $p$.

\end{proof}

Given a collection of linear forms   $\mathcal{L}= \{l_{1}, l_{2},\cdots,l_{s}\}$, we let $JM(X,\mathcal{L})$ denote the module whose matrix of generators is gotten by adding as rows $ \{l_{1}, l_{2},\cdots,l_{s}\}$ to the jacobian matrix of a set of generators of $I(X)$. In a similar way, let $(M,\mathcal{L})$ denote the module whose matrix of generators is gotten by adding as rows the $ \{l_{1}, l_{2},\cdots,l_{s}\}$ to a matrix of generators of $M$.

\begin{Prop} Let $\mathcal{L}= \{l_{1}, l_{2},\cdots,l_{s}\}$ be a collection of linear
1-forms linearly independent at the origin. Consider the hyperplanes
defined by the  forms $\sum a_{i}l_{i}$. None of these hyperplanes is a limiting
tangent hyperplane to $X,0$ at the origin if and only if
$\overline{JM(X,\mathcal{L})}= \overline{JM(X)}\oplus
\mathcal{O}_{{X}}^{s}$.

\end{Prop}

\begin{proof} It suffices to show that $JM(X,\mathcal{L})$ is a
reduction of $JM(X)\oplus\mathcal{O}_{{X}}^{s}$ if and only if
$JM(X)_{p}$ is a reduction of $JM(X)$. Suppose $JM(X,\mathcal{L})$
is a reduction of $JM(X)\oplus\mathcal{O}_{{X}}^{s}$, then
$\overline{JM(X,\mathcal{L})}$ contains $JM(X)\oplus0$. Restricting
to curves, this implies $\overline{JM(X)_{p}}$ contains $JM(X)$.

Suppose $JM(X)_{p}$ is a reduction of $JM(X)$. Then
$\overline{JM(X,\mathcal{L})}$ contains $JM(X)\oplus0$. Let $\{v_i\}$ be a collection of vectors such that $l_i(v_j)=\delta_{i,j}$, then $\overline{JM(X,\mathcal{L})}$
contains $\left(
  \begin{array}{c}
    D(F)(v_i) \\
    \mathcal{L} (v_i)\\
  \end{array}
\right)$ and $JM(X)\oplus0$, so it contains
$JM(X)\oplus\mathcal{O}_{{X}}^{s}$\end{proof}

The previous two propositions can be easily generalized using the same proof. Given $M$ a submodule of a free module $F$, $\Projan \mathcal{R}(M)$ has a canonical projection to $X$ which is a fibration over the Z-open subset $U_M$ of $X$ on which $M$ has maximal rank. The fiber of this map consists of hyperplanes. Call the planes in the fibers over $U_M$, $M$-planes. The planes in the fibers over $C(M)$ then, are limiting $M$-planes. Then the analogues of the previous two results are:

\begin{Prop} Let $X$ be a representative of the germ of a
complex analytic space, and let $\mathcal{L}= \{l_{1},
l_{2},\cdots,l_{s}\}$ be a collection of linear forms linearly independent at the origin.  Consider the hyperplanes
defined by the  forms $\sum a_{i}l_{i}$. None of these hyperplanes is a limiting
$M$-hyperplane to $X,0$ at the origin if and only if
$$\overline{M_{p}}= \overline{M},$$where p is a submersion
whose kernel is the intersection of the kernels of $l_{1}, \cdots,l_{s}$ if and only if

$$\overline{(M,\mathcal{L})}= \overline{M}\oplus
\mathcal{O}_{{X}}^{s}.$$
\end{Prop}

Let $\{\omega_{j}^{(i)}\}$ be a collection of $1$-forms on the
variety $X$, for simplicity we  will denote $S(\omega_{j}^{(i)})$ with
$1\leq i \leq s-1$ by $\mathcal{C}$. In the next proposition we are interested in characterizing those collections for which the origin is not a special point or restricted special point.

\begin{Prop} Let $(X,0)$ be the germ of an equidimensional reduced
analytic variety, $X$ a representative of the germ and $\{
\omega^{(i)}_{j}\}$ a collection of 1-forms; assume the generic point of each component of ${\mathcal{C}}$ lies in $X_{reg}$. Assume also the last collection  $\{
\omega^{(s)}_{j}\}$ is linearly independent at $0$. The origin is not a
restricted special point of the collection $\{ \omega^{(i)}_{j}\}$ iff
$JM(X,\omega^{(s)})\vert_{\mathcal{C}}$ is a reduction of
$JM(X)\vert_{\mathcal{C}}\oplus
\mathcal{O}_{\mathcal{C}}^{d-k_{s}+1}$.

If all of the collections are linearly independent at the origin, and we assume  the first $s-1$ elements are proper, then origin is not a special point of the collection $\{ \omega^{(i)}_{j}\}$ iff
$JM(X,\omega^{(s)})\vert_{\mathcal{C}}$ is a reduction of
$JM(X)\vert_{\mathcal{C}}\oplus
\mathcal{O}_{\mathcal{C}}^{d-k_{s}+1}$.

\end{Prop}

\begin{proof}

Let $\mathcal{L} = \{l_{1}, l_{2}, \cdots, l_{d-k_{s}+1}\}$ be a
collection of linear 1-forms such that $\omega^{s}_{i}(0)=l_{i}$, as
in Lemma 3.3 of \cite{Gaff-AG}, using the integral form of
Nakayama's lemma we have that $JM(X,\omega^{(s)})\vert_{\mathcal{C}}$
is a reduction of $JM(X)\vert_{\mathcal{C}}\oplus
\mathcal{O}_{\mathcal{C}}^{d-k_{s}+1}$ iff
$JM(X,\mathcal{L})\vert_{\mathcal{C}}$ is a reduction of
$JM(X)\vert_{\mathcal{C}}\oplus
\mathcal{O}_{\mathcal{C}}^{d-k_{s}+1}$. So we can work with
$\mathcal{L}$.

Now we apply the previous proposition, where $M=JM(X)$ restricted to ${\mathcal{C}}$. Then the limiting $M$-hyperplanes are just the tangent hyperplanes to $X$ as the generic point of each component of ${\mathcal{C}}$ is in $X_{reg}$. If some combination of the $\omega^{s}_{i}(0)=l_{i}$ are a limiting tangent hyperplane to $X$, then that combination is zero when restricted to the limiting tangent plane, and the collection is linearly dependent.

If we assume properness, then since every special point is a restricted special point, the result follows.

\end{proof}

We will need a refinement of this result for later. The key point in the above argument, is that  $JM(X,\omega^{(s)})\vert_{\mathcal{C}}$ is a reduction of
$JM(X)\vert_{\mathcal{C}}\oplus
\mathcal{O}_{\mathcal{C}}^{d-k_{s}+1}$ if and only if  none of the hyperplanes defined by $\{\omega^{(s)}(0)\}$ is a limiting
tangent hyperplane to $X,0$ at the origin along curves on  ${\mathcal{C}}$. Given the collection $\{\omega^{(s)}\}$ , we can deform it to $\{\omega^{(s)}(0)\}$ by using the linear deformation. This fixes the one jet of the collection. Denote this family of collections by $\{\omega^{(s)}_L\}$ . Denote the family of  sections defined by fixing the first $s-1$ collections and deforming the last using the linear deformation by $\Gamma_{\{\omega_{s-1,L}\}}$

\begin{Prop} Assume $JM(X,\omega^{(s)})\vert_{\mathcal{C}}$ is a reduction of
$JM(X)\vert_{\mathcal{C}}\oplus
\mathcal{O}_{\mathcal{C}}^{d-k_{s}+1}$. Then $Im(\Gamma_{\{\omega_{s-1,L}\}})\cdot \mathbb{D}_{X}^{k}$ is constant in the linear deformation.
\end{Prop}
\begin{proof} Since $JM(X,\omega^{(s)})\vert_{\mathcal{C}}$ is a reduction of
$JM(X)\vert_{\mathcal{C}}\oplus
\mathcal{O}_{\mathcal{C}}^{d-k_{s}+1}$, the same is true for any member of the family $\{\omega^{(s)}_L\}$. Suppose for some parameter value $t_0$  that the intersection number changes, ie a point splits off. This gives a curve of points in ${\mathcal{C}}$, where at each point $p$, a member of $\{\omega^{(s)}_L\}$,  is degenerate when restricted to some plane which is a point over $p$ in the Nash modification. This implies that $\{\omega^{(s)}_L\}(t_0)$ is degenerate when restricted to some plane which is a point over $0$ in the Nash modification. As this plane can be reached through a curve on $C$, it contradicts  that $JM(X,\omega^{(s)}_L(t_0))\vert_{\mathcal{C}}$ is a reduction of
$JM(X)\vert_{\mathcal{C}}\oplus
\mathcal{O}_{\mathcal{C}}^{d-k_{s}+1}$.
\end{proof}

\begin{definition}
$H_{d-1}(X),$ by definition, consists of all elements of
$\mathcal{O}^{p}_{X}$ which are in the integral closure of $JM(X^d)$ except at the origin. A related idea is $H_{c-1}(X,C^c)$ where $C$ is a subset of $X$ of dimension $c$. It consists of all elements of $\mathcal{O}^{p}_{C}$ which are in the integral closure of  $JM(X)$ restricted to $C$ except at the origin.

In general, $H_{i}(X)$ consists of all elements of
$\mathcal{O}^{p}_{X}$ which are in the integral closure of $JM(X)$ off a set of codimension
$i + 1.$ Sometimes we write $H_{i}(JM(X)).$ The meaning of  $H_{i}(X,C^c)$ is similar.

\end{definition}
\begin{Prop} Let $(X,0)$ be the germ of an equidimensional reduced
analytic variety, $X$ a representative of the germ and $\{
\omega^{(i)}_{j}\}$ a collection of 1-forms; assume the generic point of each component of ${\mathcal{C}}$ lies in $X_{reg}$. Assume also the last collection  $\{
\omega^{(s)}_{j}\}$ is linearly independent at $0$. The origin is at most an isolated restricted special point of the collection
 if and only if  $JM(X, \mathcal{L})\vert_{\mathcal{C}}$ is a reduction
of $H_{c-1}(X,C)\vert_{\mathcal{C}}\oplus
\mathcal{O}_{\mathcal{C}}^{d-k_{s}+1}$ except possibly at x.

If in addition, the first $s-1$ collections are proper on a deleted neighborhood of the origin, then the origin is at most an isolated special point.
\end{Prop}

\begin{proof} Suppose the origin is an isolated restricted special point. Let $U$ be a neighborhood of $0$ in $X$ such that $x$ is the only restricted special point. Then by proposition 4.7,
$\Gamma(\omega^{s})$ misses $T^{*}(X)\vert_{\mathcal{C}}$ on $U
\setminus \{0\}$.

 Then $JM(X, \mathcal{L})\vert_{\mathcal{C}}$ is a
reduction of $JM(X)\vert_{\mathcal{C}}\oplus
\mathcal{O}_{\mathcal{C}}^{d-k_{s}+1}$ at all $x\in U$, $x\ne 0$ by the previous proposition.

Hence by definition it is a reduction of
$H_{c-1}(X, C)\vert_{\mathcal{C}}\oplus
\mathcal{O}_{\mathcal{C}}^{d-k_{s}+1}$ except possibly at $x$.

On the other hand assume the reduction criterion holds at each point of $U-0$. This implies  $JM(X, \mathcal{L})\vert_{\mathcal{C}}$ is a
reduction of $JM(X)\vert_{\mathcal{C}}\oplus
\mathcal{O}_{\mathcal{C}}^{d-k_{s}+1}$ as this last module is a submodule of $H_{c-1}(X,C)\vert_{\mathcal{C}}\oplus
\mathcal{O}_{\mathcal{C}}^{d-k_{s}+1}$. This implies that there are no restricted special points on $U$ except possibly the origin.

If in addition, the first $s-1$ collections are proper on a deleted neighborhood of the origin, then the lack of restricted special points on $U-0$ is equivalent to a lack of special points.

\end{proof}
The last proposition leaves open the question as to whether the origin is a restricted special point if the reduction criterion holds. The next proposition lets us settle this point.

\begin{Prop}\label{soma} Suppose the origin is at most an isolated restricted special point. Then $$e(JM(X, \omega^{s})\vert_{\mathcal{C}}, JM(X)\vert_{\mathcal{C}} \oplus \mathcal{O}_{\mathcal{C}}^{d-k_{s}+1}, 0)$$
$$= e(JM(X, \omega^{s})\vert_{\mathcal{C}},H_{c-1}(X,{\mathcal{C}}) \oplus \mathcal{O}_{\mathcal{C}}^{d-k_{s}+1}, 0)$$ $$ - e(JM(X,\mathcal{L})\vert_{\mathcal{C}},H_{c-1}(X,{\mathcal{C}})\oplus\mathcal{O}_{\mathcal{C}}^{d-k_{s}+1}, 0)$$ where
$\mathcal{L}$ is a collection of linear 1-forms such that $0$ is not
a restricted special point for it.

The origin is not a restricted special point if and only if
$$e(JM(X, \omega^{s})\vert_{\mathcal{C}},H_{c-1}(X,{\mathcal{C}}) \oplus \mathcal{O}_{\mathcal{C}}^{d-k_{s}+1}, 0)$$
$$=  e(JM(X,\mathcal{L})\vert_{\mathcal{C}},H_{c-1}(X,{\mathcal{C}})\oplus\mathcal{O}_{\mathcal{C}}^{d-k_{s}+1}, 0)$$
where
$\mathcal{L}$ is a collection of linear 1-forms such that $0$ is not
a restricted special point for it.

\end{Prop}

\begin{proof} Since the origin is at most an isolated restricted special point all three multiplicities are well defined. Then, the proof is based on a fundamental result due to Kleiman and
Thorup, the principle of additivity \cite{KT}. Given a sequence of
$\mathcal{O}_{X}$ modules $A \subset B \subset C$ such that the
multiplicity of the pairs is well defined,
then$$e(A,C)=e(A,B)+e(B,C).$$The result follows by setting $A= JM(X,
\omega^{s})\vert_{C}$,
$B=JM(X)\vert_{C}\oplus\mathcal{O}_{C}^{d-k_{s}+1}$,
$C=H_{d-1}(X,{\mathcal{C}})\oplus\mathcal{O}_{C}^{d-k_{s}+1}$. Using the
fact that $0$ is not a restricted special point for $\mathcal{L}$ we have that
the multiplicity of $(JM(X,
\mathcal{L})\vert_{C},H_{d-1}(X)\vert_{C}\oplus\mathcal{O}_{C}^{d-k_{s}+1})$
and
$(JM(X)\vert_{C}\oplus\mathcal{O}_{C}^{d-k_{s}+1},H_{d-1}(X)\vert_{C}\oplus\mathcal{O}_{C}^{d-k_{s}+1})$
are the same.

The origin is not a restricted special point by 5.5 if and only if  $JM(X,
\omega^{s})\vert_{C}$ is a reduction of $JM(X)\vert_{C}\oplus\mathcal{O}_{C}^{d-k_{s}+1}$. The reduction statement holds at $0$ if and only if $e(JM(X, \omega^{s})\vert_{\mathcal{C}},H_{c-1}(X,{\mathcal{C}}) \oplus \mathcal{O}_{\mathcal{C}}^{d-k_{s}+1}, 0)=0$, which is true if and only if
$$e(JM(X, \omega^{s})\vert_{\mathcal{C}},H_{c-1}(X,{\mathcal{C}}) \oplus \mathcal{O}_{\mathcal{C}}^{d-k_{s}+1}, 0)$$
$$=  e(JM(X,\mathcal{L})\vert_{\mathcal{C}},H_{c-1}(X,{\mathcal{C}})\oplus\mathcal{O}_{\mathcal{C}}^{d-k_{s}+1}, 0)$$
where
$\mathcal{L}$ is a collection of linear 1-forms such that $0$ is not
a restricted special point for it.

\end{proof}

Our next step to apply The Multiplicity Polar Theorem is to show that
the polar curve  of $JM(X,\omega^{(s)}+tM)\vert_{\mathcal{C}}$ is
empty.

\begin{Prop} Let $X^{d},0 \subset \mathbb{C}^{n},0$ and $\{\omega\}$, a
collection of $1$-forms $\{\omega\}$ $1\leq i \leq s$, $1\leq j \leq
d-k_{i}+1$, $\sum k_{i}=d$.  Assume further the collection has an isolated singularity
at the origin, and that the generic point of $S(\omega^{(i)}_{j})$,
$1\leq i \leq s-1$ is in $X_{reg}$.
 Then, the polar curve of the module $JM(X,\omega^{(s)}+tM)\vert_{\mathcal{C}}$ is empty for
$\mathfrak{C}= \mathcal{C} \times \mathbb{C}$, where $M$ is a
collection of generic linear forms.
\end{Prop}

\begin{proof}
The polar variety of codimension $k$ of $M$ in $X$ denoted
$\Gamma_{k}(M)$ is constructed by intersecting $\Projan
\mathcal{R}(M)$ with $X \times H_{g+k-1}$ where $H_{g+k-1}$ consists of the set of hyperplanes which contain a
general plane of dimension $g+k-1$, and $g$ is the generic rank of
$JM(X,\omega^{(s)}+tM)$, then projecting to $X$.   Note that if $M$ has $n$ generators, so that $\Projan
\mathcal{R}(M)$ is contained in $X\times \mathbb P^{n-1}$, and the dimension of  $\Projan
\mathcal{R}(M)$ is greater than or equal to $n+1$ then the polar varieties of $M$ of codimension $n$ or more are empty, because the codimension of a point in $\mathbb P^{n-1}$ is $n-1$.

With this observation in mind,  the next step is to  compute the dimension of  $\Projan
\mathcal{R}(JM(X,\omega^{(s)}+tM)\vert_{\mathcal{C}})$.
This dimension  is the
dimension of the base plus the generic rank of
$\left(
  \begin{array}{c}
    D(F) \\
    \omega \\
  \end{array}
\right)$ minus 1. Now the generic rank of the jacobian matrix is $n-d$, while the generic rank of the jacobian matrix augmented by the $\{\omega_j^s\}$ is $(n-d)+(d-k_s+1)=n-k_{s}+1$. This follows because the generic point of ${\mathcal{C}}$ is a smooth point of $X$ hence the jacobian matrix has maximal rank there. Because $0$ is an isolated singularity, it follows that the augmented matrix generically has maximal rank. Thus we have, since $g=n-k_{s}+1$,

$\dim \Projan \mathcal{R}(JM(X, \omega^{(s)}+tM)\vert_{\mathcal{C}})=
k_{s}+1+(n-k_{s}+1))-1=n+1.$

Since the dimension of $\Projan \mathcal{R}(JM(X,
\omega^{(s)}+tM)\vert_{\mathcal{C}})$ is greater than or equal to the
number of generators, there is no polar curve for $JM(X,
\omega^{(s)}+tM$).

\end{proof}

%computation on Morse points

\begin{Prop} Suppose $X$ is smooth and $\omega$ is a $1$-form such that  $\omega$ has a Morse point at $0$ then
$e(JM(X,\omega),JM(X)\oplus \mathcal{O}_X,0)=1$.
\end{Prop}

\begin{proof}
Since $X$ is a smooth manifold, the number of
equations of $X$ is $n-d$, so the matrix of generators of
$JM(X,\omega)$ has $n-d+1$ rows, $n$ columns, and a matrix of
generators of $JM(X)\oplus \mathcal{O}_X$ also has $n-d+1$ rows with
the same $n-d$ first rows. We may assume the equations for $X$ are
$z_1=\dots=z_{n-d}=0$.

Then the Rees algebra of $JM(X)\oplus \mathcal{O}_X$ is
$\mathcal{O}_X[S_1, \dots, S_{n-d}, S_{n-d+1}]$, while the ideal
corresponding to the inclusion of the Rees algebra of $JM(X,\omega)$
in that of $JM(X)\oplus \mathcal{O}_X$ is $(S_i, {z_j}S_{n-d+1})$ where $1\le i\le n-d$, $n-d<j\le n$.
Now in this example, we know that
$1=e(JM(X,f),\mathcal{O}^{n-d+1}_X)=e(JM(X,f), JM(X)\oplus
\mathcal{O}_X)+ e(JM(X)\oplus \mathcal{O}_X,
\mathcal{O}^{n-d+1}_X)$, while $ e(JM(X)\oplus \mathcal{O}_X,
\mathcal{O}^{n-d+1}_X)=0$ since the two modules are the same. This
uses the additivity of the multiplicity, the fact that $\omega$ is
Morse on $X$, and the fact that the multiplicity of
$e(JM(X,\omega),\mathcal{O}^{n-d+1}_X)$ is the colength of its ideal
of maximal minors.

Now we want to show that we get the same result even if the number of equations is larger than $n-d$. (This happens for example, if we are working at a smooth point of a space which is singular at the origin.) Suppose our choice of
generators for $I(X)$ has $p$ generators, $p\ge n-d$. By a change of
coordinates we can assume the equations have the form
$x_1=\dots=x_{n-d}=g_{n-d+1}=\dots=g_p =0$, where the matrix of
generators of $JM(X)$ must have the last $d$ columns $0$. Then the
Rees alegbra of $JM(X)\oplus \mathcal{O}_X$  is the same as before,
as is the ideal induced by $JM(X,\omega)$, so the multiplicity of
the pair is the same.
\end{proof}

\begin{Prop} Let  $\{\omega^{(i)}_{j}\}$ be a collection of 1-forms, with  $1\leq i \leq s$, $1\leq j \leq
d-k_{i}+1$, $\sum k_{i}=d$ such that,
restricted to $X^d$, $\{\omega^{(s)}\}$ has a non-degenerate special point at $x$, $x$ a smooth point of ${\mathcal C}$ and $X$. Then
$$e(JM(X,\omega^{(s)})\vert_{\mathcal C},JM(X)\vert_{\mathcal C}\oplus
\mathcal{O}_{{\mathcal C}^{n-k_{s}+1}},x)=1.$$
\end{Prop}

\begin{proof}
Let us suppose, that $X$ is a smooth manifold and the number of
equations of $X$ is $n-d$ so that the matrix of generators of
$JM(X,\omega^{(s)})$ has $n-k_{s}+1$ rows, $n$ columns, and a matrix of
generators of $JM(X)\oplus \mathcal{O}^{d-k_{s}+1}_{\mathcal C}$ also has
$n-k_{s}+1$ rows with the same $n-d$ first rows. We may assume the
equations for $X$ are $x_1=\dots=x_{n-d}=0$. Since we assume $\mathcal C$ is smooth at $x$, and it has dimension $k_s$, assume that the last $k_s$ coordinates on $X$ define $\mathcal C$. We may assume that the collection $\omega^{(s)}$ has form $\{ dx_{n-d+i}, dh\}$ where $1\le i\le d-k_s$ and $h=\sum_{j=1}^{k_s} x_{n-k_s+j}^2$.

As in the last Proposition, in this example, we know that
$$1=e(JM(X,\omega^{(s)})\vert_{C},\mathcal{O}^{n-k_{s}+1}_{C})$$
$$=e(JM(X,\omega^{(s)}),
JM(X)\vert_{C}\oplus \mathcal{O}_{C})+ e(JM(X)\vert_{C}\oplus
\mathcal{O}^{d-k_{s}+1}_C), \mathcal{O}^{n-k_{s}+1}_{C}),$$ while $
e(JM(X)\vert_{C}\oplus \mathcal{O}^{d-k_{s}+1}_{C},
\mathcal{O}^{n-k_{s}+1}_{C})=0$ since the two modules are the same.
This uses the additivity of the multiplicity, the fact that $\omega$
is non-degenerate on $X$, and the fact that the multiplicity of
$e(JM(X,\omega^{(s)}),\mathcal{O}^{n-k_{s}+1}_X)$ is the colength of its
ideal of maximal minors, and as in the last Proposition, the general
result follows.

\end{proof}

Before giving our main result, it is useful to consider the difference between the case of a vector bundle well defined at all points, and a bundle like the tangent bundle to a singular space which is not well defined at $S(X)$. In the second case, we get a special point if  $\widetilde{X}(0)\cap B(\omega)$ is non-empty. If we alter the last collection of forms, then we can make the last collection generic on ${\mathcal{C}}$, but the singular locus of the modified forms may still be non-empty. In this case the intersection number $Im(\Gamma_{\widetilde\omega})$ and $\mathbb{D}_{X}^{k}$ may still be non-zero at the origin.

In the first case, the analogue of  $\widetilde{X}(0)$ consists of a single point, so by altering the last collection of forms we can ensure that the intersection number $Im(\Gamma_{\widetilde\omega})$ and $\mathbb{D}_{X}^{k}$ is zero at the origin.

This phenomena is the reason that the formula for the Chern numbers for the Nash bundle has many terms, while that of a vector bundle on $X$ has only one.

The next theorem is the key step in the proof of our main result. It allows us to fix each of the collections in turn, until we are left with collections which are linearly independent at the origin and which are proper. Of course, this last collection has no special points.

%%%%%%%%%%%%%%%%%%%%%%%%%%%%%%%%%%%%%%%%%%%%%%%%%%%%%%%%%%%%%%%%%%%%%
Let $X^{d},0 \subset \mathbb{C}^{n},0$ and

\begin{theorem}\label{index} Let $(X^d,0) \subset (\mathbb{C}^{n},0)$ be the germ of an
equidimensional reduced analytic variety, with representative
$X$, $\{\omega^{(i)}_{j}\}$, a
collection of $1$-forms with $1\leq i \leq s$, $1\leq j \leq
d-k_{i}+1$, $\sum k_{i}=d.$ Assume further the collection has an isolated singularity
at the origin, and that the generic point of $S(\omega^{(i)}_{j})$
$1\leq i \leq s-1$ is in $X_{reg}$.  We have that,

$$Ch_{X,0}\{\omega^{(i)}_{j}\}=
e(JM(X,\omega^{(s)})\vert_{\mathcal{C}},
JM(X)\vert_{\mathcal{C}}\oplus\mathcal{O}_{\mathcal{C}}^{d-k_{s}+1},
0)+ Im(\Gamma_{\widetilde\omega})\cdot \mathbb{D}_{X}^{k}.$$

\end{theorem}

\begin{proof} Let us consider the family of sets ${\mathcal{C} }\times \mathbb{C}$.
Let $\pi_{{\mathcal{C}}}$ denote the projection from ${\mathcal{C} }\times \mathbb{C}$ to
$\mathcal{C}$, and $\pi_{t}$ the projection to $\mathbb{C}$. By conservation
of number and taking M as in the Proposition \ref{proposicao},
$\Gamma(\omega^{s})\cdot T^{*}(X)\vert_{{\mathcal{C}}}$ is just
$$\Gamma(\omega^{s}+ tM)\cdot T^{*}(X)\vert_{\mathcal{C}}$$ for $t$ close to $0$,
and this is just the number of non degenerate special points of
$\{(\omega)_{1\leq i\leq s-1}, \omega^{s}+tM\}$ for $t\ne 0$, and the intersection number  $Im(\Gamma_{\widetilde\omega})\cdot \mathbb{D}_{X}^{k}.$ (Recall that the collection ${\widetilde\omega}$ was defined before Cor 4.12.)
%%%
To show that the Multiplicity Polar theorem applies, we must also show that $$C(\Projan (\mathcal{R}(JM(X\times\mathbb{C},
\omega_{t})_{\pi_{t}}\vert_{{\mathcal{C}}\times \mathbb{C}})))(0) = C(\Projan (\mathcal{R}(JM(X,
\omega)\vert_{\mathcal{C}}))),$$ except possibly over $(0, 0) \in {\mathcal{C} } \times 0$.

Since $N = JM(X)\vert_{\mathcal{C}\times C}\mathcal{O}_{\mathcal{C}\times C}\oplus
\mathcal{O}_{\mathcal{C}\times C}^{d-k_{s}+1}$ as a family of modules is
independent of $t$, $\Projan R(N)$ is a product, hence
$C(\Projan \mathcal{R}(N))(0) = C(\Projan \mathcal{R}(N(0)))$. Now,  at any
point p of ${\mathcal{C} } \times 0$ close to the origin, there exists a
neighborhood $U$ of $p$ such that on $U$, $\overline{JM(X\times\mathbb{C},
\omega_{t})_{\pi_{t}}\vert_{{\mathcal{C}}\times \mathbb{C}}} = N$. This implies that over $U$,
$\Projan \mathcal{R}(N)$ is finite over $\Projan \mathcal{R}(JM(X\times\mathbb{C},
\omega_{t})_{\pi_{t}}\vert_{{\mathcal{C}}\times \mathbb{C}}),$ and on $U \cap C \times 0,
\Projan \mathcal{R}(N(0))$ is finite over $\Projan\mathcal{R}(JM(X, \omega_{0})\vert_{C}).$
Now, since $\Projan R(JM(X, \omega_{0})\vert_{C}) \subset
\Projan \mathcal{R}(JM(X\times\mathbb{C},
\omega_{t})_{\pi_{t}}\vert_{{\mathcal{C}}\times \mathbb{C}})(0),$ the desired
equality follows, for any element of $\Projan \mathcal{R}(JM(X\times\mathbb{C},
\omega_{t})_{\pi_{t}}\vert_{{\mathcal{C}}\times \mathbb{C}} )(0)$ has a preimage in
$\Projan \mathcal{R}(N)(0)$ which is $\Projan \mathcal{R}(N(0)),$ and the last set maps to
$\Projan \mathcal{R}((JM(X, \omega_{0})\vert_{C})).$ So, the multiplicity polar
theorem applies. Note, that since $\Projan \mathcal{R}(N)$ is a product, N has
no polar curve, and by 5.10 we know that $JM(X\times\mathbb{C},
\omega_{t})_{\pi_{t}}\vert_{{\mathcal{C}}\times \mathbb{C}} )$ has no polar curve either.
%%%%
Now, by Proposition \ref{proposicao}  we have,
$Ch_{X,0}(\omega_{j}^{(i)})=\Gamma(\omega^{s}+ tM)\cdot
T^{*}(X)\vert_{C}$.

Then, using the Multiplicity Polar Theorem we have,

$$Ch_{X,0}(\omega_{j}^{(i)})=e(JM(X,\omega^{(s)})\vert_{\mathcal{C}},
JM(X)\vert_{\mathcal{C}}\oplus\mathcal{O}_{\mathcal{C}}^{d-k_{s}+1},
0)+ Im(\Gamma_{\widetilde\omega})\cdot \mathbb{D}_{X}^{k}.$$
\end{proof}

 Suppose the collection is ordered so that the first $r$ collections meet properly and $r$ is the largest integer for which this is true. Let $C_i$ denote $C(\omega^{(1)},\dots,\omega^{(i)}, \widetilde{\omega^{(s)}},\dots,  \widetilde{\omega^{(i+2)}} ) $ where $i\le s-1$, and  $\widetilde{\omega^{(j)}}$ is a collection of generic linear forms so that the collections $\{\omega^{(i)}\}, i\le r$, $\{\widetilde{\omega^{(j)}}\}$ meet properly.

\begin{cor}\label{formula} Suppose the collection is ordered so that the first $r$ collections meet properly and $r$ is the largest integer for which this is true. In the setup of the last Theorem we have that
$$Ch_{X,0}(\omega)=\sum_{i=r}^{i=s-1} e(JM(X,\omega^{(i+1)})\vert_{\mathcal{C}_i},
JM(X)\vert_{\mathcal{C}_i}\oplus\mathcal{O}_{\mathcal{C}_i}^{d-k_{i+1}+1},
0)$$
\end{cor}

\begin{proof} We prove the Corollary by applying the previous theorem multiple times. First, to $\{\omega(s)|_ {\mathcal{C}_{s-1}}\}$, then to $\{\omega(s-1)|_ {\mathcal{C}_{s-2}}\}$. Finally, when all but one of our collections meet properly, applying the theorem to $\{\omega(r+1)|_ {\mathcal{C}_{r}}\}$ produces only a single term as the intersection number term is $0$.

\end{proof}

%:
Let $C_{i'}$ denote $C( \widetilde{\omega^{(s)}},\dots,  \widetilde{\omega^{(i+2)}} ) $. Then, $C_{i'}$is related to the polar varieties of $X$.  For $C( \widetilde{\omega^{(i+2)}})$ is the polar variety of codimension $k(i+2)$, so $C_{i'}$ is the intersection of the corresponding polar varieties. If $X$ is a hypersurface, then  in fact this is the polar variety of codimension $\sum_{i+2}^s k(j)$. The hypersurface case is special because since $TX_x$, $x\in X_0$ has codimension 1, the kernels of all of the $ \widetilde{\omega^{(j)}}$ are contained in $TX_x$ if $x\in C_{i'}$, hence $x$ is in the polar defined by the union of the kernels.

\begin{cor} \label{expansion} In the set up of the last proposition we have
$$Ch_{X,0}(\omega)=\sum_{i=r}^{i=s-1} e(JM(X,\omega^{(i+1)})\vert_{\mathcal{C}_i},
JM(X)\vert_{\mathcal{C}_i}\oplus\mathcal{O}_{\mathcal{C}_i}^{d-k_{i+1}+1},
0)$$
$$= \sum_{i=r}^{i=s-1} e(JM(X, \omega^{i+1})\vert_{\mathcal{C}_i},H_{c_i-1}(X,{\mathcal{C}_i}) \oplus \mathcal{O}_{\mathcal{C}_i}^{d-k_{i+1}+1}, 0)$$ $$ - e(JM(X,\widetilde{ \omega^{i+1}})\vert_{\mathcal{C}_i},H_{c_i-1}(X,{\mathcal{C}_i})\oplus\mathcal{O}_{\mathcal{C}_i}^{d-k_{i+1}+1}, 0).$$\end{cor}

\begin{proof} Apply Propositon 5.9 to expand $e(JM(X,\omega^{(i+1)})\vert_{\mathcal{C}_i},
JM(X)\vert_{\mathcal{C}_i}\oplus\mathcal{O}_{\mathcal{C}}^{d-k_{i+1}+1},
0)$.

\end{proof}

We want to specialize our results to the case where $X^d,0\subset \mathbb C^n$ is an ICIS to compare with those of \cite{EG}. Given a collection of holomorphic forms $\omega$ with an isolated singular point at $0$, Ebeling and Gussein-Zade define another notion of index in \cite{EG}. In the case $X,0$ is an ICIS, the index amounts to smoothing $X$ as well as making the forms general, then counting the number of singular points of the new collection on the smoothing. This index can be calculated as follows: suppose ${\omega_j^{(i)}}$, $1\le i\le s$, $1\le j\le n_{k_i}+1$, $\sum k_i=d$, augment the jacobian matrix of $X$ for each $i$ with ${\omega_j^i}$, producing $s$ matrices. Form an ideal in ${\mathcal  O}_n$, using as generators, the generators of $I(X)$, and the maximal minors of the augmented matrices. Denote the resulting ideal by $I_{X,{\omega_j^{(i)}}}$. Then the index, denoted $ind_{X,0}(\{\omega\})$ is just the colength of $I_{X,{\omega_j^{(i)}}}$
in ${\mathcal  O}_n$. (\cite {EG}, Theorem 20.) Using this index they show that

$$Ch_{X,0}=ind_{X,0}(\{\omega\})-ind_{X,0}(\{l\})$$
where $l=\{l^{(i)}_j\}$ is a generic collection of forms. (\cite {EG} Cor. 4.)

We will see that this formula can be recovered from the last corollary. If $X$ is an ICIS, and the ${\mathcal{C}_i}$  have the minimal dimension then  the ${\mathcal{C}_i}$ are Cohen-Macaulay, with ideal the ideal of $X$ and the maximal minors of the augmented matrices. Further, the  matrix of generators of   $JM(X)$  has maximal rank except at the origin when restricted to ${\mathcal{C}_i}$. This implies that
$H_{c_i-1}(X,{\mathcal{C}_i})\vert_{\mathcal{C}_i}$ is free, so
$e(JM(X, \omega^{i+1})\vert_{\mathcal{C}_i},H_{c_i-1}(X,{\mathcal{C}_i}) \oplus \mathcal{O}_{\mathcal{C}_i}^{d-k_{i+1}+1}, 0)=e(JM(X, \omega^{i+1})\vert_{\mathcal{C}_i}.$
Since ${\mathcal{O_C}_i}$ is Cohen-Macaulay, the last multiplicity is just the colength of the ideal formed by the maximal minors of the augmented matrices formed from $JM(X)$, the collection of forms used to define ${\mathcal{C}_i}$ and $\omega^{i+1}$. It follows that
$$e(JM(X, \omega^{i+1})\vert_{\mathcal{C}_i}=ind( \{\omega^{(1)},\dots,\omega^{(i)} ,\omega^{i+1}, \widetilde{\omega^{(s)}},\dots,  \widetilde{\omega^{(i+2)}}\}).$$

Now the sum on the right hand side of the last corollary telescopes to
$$Ch_{X,0}=ind_{X,0}(\{\omega\})-ind( \{\omega^{(1)},\dots,\omega^{(r)},\widetilde{\omega^{(s)}},\dots,  \widetilde{\omega^{(r+1)}}\}).$$
Since the last collection on the right hand side is proper, an argument similar to that of Prop. 5.5 shows that the last term is $ind_{X,0}(\{l\})$.

In the case of surfaces it is not hard to compute with our formula, and we give some examples. As preparation we give two versions of our formula for the case of surfaces which are not ICIS.

The general case of our theorem becomes:

\begin{cor} Let $(X,0) \subset (\mathbb{C}^{n},0)$ be the germ of a
purely 2 dimensional reduced analytic variety, with representative
$X$, $\{\omega\}$, a
collection of sets of $1$-forms $\{\omega\}$ $1\leq i \leq 2$, each with two elements.  Assume further the collection has an isolated singularity
at the origin, and that the generic point of $S(\omega^{(1)}_{j})$
is in $X_{reg}$.  We have that,

$$Ch_{X,0}\{\omega^{(i)}_{j}\}=
e(JM(X,\omega^{(2)})\vert_{\mathcal{C}},
JM(X)\vert_{\mathcal{C}}\oplus\mathcal{O}_{\mathcal{C}}^{d-k_{s}+1},
0)+ Im(\Gamma_{\widetilde\omega})\cdot \mathbb{D}_{X}^{k}.$$

Further,
$$Ch_{X,0}\{\omega^{(i)}_{j}\}= e(JM(X, \omega^{2})\vert_{\mathcal{C}_1},H_{0}(X,{\mathcal{C}_1}) \oplus \mathcal{O}_{\mathcal{C}_1}^{n-d}, 0) $$
 $$-e(JM(X,\widetilde{ \omega^{2}})\vert_{\mathcal{C}_1},H_{0}(X,{\mathcal{C}_1})\oplus\mathcal{O}_{\mathcal{C}_1}^{n-d}, 0)
+e(JM(X, \omega^{1})\vert_{\mathcal{C}_0},H_{0}(X,{\mathcal{C}_0}) \oplus \mathcal{O}_{\mathcal{C}_0}^{n-d}, 0) $$
$$ -e(JM(X,\widetilde{ \omega^{1}})\vert_{\mathcal{C}_0},H_{0}(X,{\mathcal{C}_0})\oplus\mathcal{O}_{\mathcal{C}_0}^{n-d}, 0).$$
\end{cor}
\begin{proof} This is Theorem 5.13 and Corollary 5.15 for the surface case.
\end{proof}
Further simplification is possible, if $X$ is a complete intersection.
\begin{cor}  Suppose in addition $X$ is a complete intersection. Then
$$Ch_{X,0}\{\omega^{(i)}_{j}\}= e(JM(X, \omega^{2})\vert_{\mathcal{C}_1}, 0)
 -e(JM(X,\widetilde{ \omega^{2}})\vert_{\mathcal{C}_1},0)$$
$$+e(JM(X, \omega^{1})\vert_{\mathcal{C}_0}, 0)
 -e(JM(X,\widetilde{ \omega^{1}})\vert_{\mathcal{C}_0}, 0)$$
 $$=\Gamma^1(\omega^1)\cdot\Gamma^1(\omega^2)-\Gamma^1(\widetilde \omega^1)\cdot\Gamma^1(\widetilde \omega^2)$$
\end{cor}
\begin{proof}  The first equality holds because $X$ is a complete intersection, and the generic point of $\mathcal{C}_i$ lies in the regular part of $X$, the Jacobian module of $X$ has maximal rank off the origin, so
$H_{0}(X,{\mathcal{C}_1}) \oplus \mathcal{O}_{\mathcal{C}_1}^{n-d}$ is just $ \mathcal{O}_{\mathcal{C}_1}^{n}$ so the multiplicity of this pair is just the ordinary Buchsbaum-Rim multiplicity. Since our curves are reduced their rings are Cohen Macaulay, so the multiplicity of $JM(X, \omega^{2})\vert_{\mathcal{C}_1}$ at $0$ is just the colength of the ideal generated by the determinant of the matrix of generators of $JM(X, \omega^{2})\vert_{\mathcal{C}_1}$. This determinant on $X$ defines the union of $S(X)$ and
$\Gamma^1(\omega^2)$ since it does so generically. Thus, this colength is just the intersection of $\Gamma^1(\omega^1)$ with $\Gamma^1(\omega^2)$  and $S(X)$. Applying this insight to each of the terms of the first equality and canceling terms results in the next equality.
\end{proof}
We give an example using this result.

\begin{exa}
Let $(X,0) \subset \mathbb{C}^{3}$ be the germ of a singular surface defined by $f: \mathbb{C}^{3},0 \to \mathbb{C},0$ be $f(x,y,z)= y^{2}-{x}^{3}$.

Take the collection of $1$-forms $\omega = \{\omega^{1}, \omega^{2}\}$ where $\omega^{1}=\{(0,x^{3},z^{2}), (z^{3},0,x^{2})\}$, and $\omega^{2}=\{(y^{2}, z^{3},0), (0,y^{3},z^{2})\}.$  Then the  local Chern obstruction of this collection is $47$. \end{exa}
We will show this using the second equality in the last Corollary. The matrix of generators of $JM(X, \omega^{1})$ and $JM(X, \omega^{2})$
are respectively,
$$\begin{pmatrix} -3x^{2} & 2y & 0 \\ 0 & x^{3} & z^{2}\\ z^{3} & 0 & x^{2}\end{pmatrix},\begin{pmatrix}  -3x^{2} & 2y & 0 \\ y^2 & z^3 & 0\\ 0& y^3& z^2\end{pmatrix}$$

Calculating the determinants of the matrix of generators of $JM(X, \omega^{2})$ and $JM(X, \omega^{1})$, we get $z^2(2y^3+3x^2z^3)$ and $-3x^7+2z^5y$. Since we are only interested in the polar curves of $\omega^i$, we use the defining equation to get the equivalent forms $z^2x^2(2xy+3z^3)$ and $y(-3xy^3+2z^5)$. So the equations of the polar curves of our collection are $ z^2(2xy+3z^3)=0$ and $-3xy^3+2z^5=0$. To calculate the intersection multiplicity, pull back to the normalization using the map $n(t,z)=(t^2,t^3,z)$. So we want the intersection multiplicity of $ z^2(2t^5+3z^3)=0$ and $-3t^{11}+2z^5=0$, which is (2)(11)+25=47. Since our underlying space is Whitney equisingular, the polar curves of $X$ are empty, so the term we have computed is the only term in the corollary, so the local Chern obstruction of this collection is $47$. (Notice that in this example, one component of the polar of $JM(X, \omega^{2})$  is not reduced. Nonetheless, a careful reading of the proof of our main result shows that in this simple case the main result continues to hold.)

We describe briefly how the work of this section can be generalized. Start with an analytic space $X$, and a bundle $E^k$ defined on a Zariski open, everywhere dense  subset $U$ of $X$, $E^k$ a sub-bundle with $k$ dimensional fiber of a bundle, $F^l$, where $F^l$ is defined everywhere. Form the relative Nash transformation $N(X,E,F)$ of $X$ as follows: form the bundle over $X$ of $k$ planes in the fiber of $F$, consider the image of the section of this bundle formed from the fibers of $E^k$, and take its closure. The relative Nash transformation has a  canonical bundle $\xi$ on it  which is a sub-bundle of the pullback of  $F^l$ to  $N(X,E,F)$,  $\xi$ and the pullback of $E$ to $N(X,E,F)$ agree restricted to $U$. By construction and restriction, sections of $F^*$ give sections of $E^*|U$, and $\xi^*$. If a collection of sections of $E^*|U$ arise in this way from a collection of sections of $F^*$, and the collection has an isolated special point at $x\in X$, then we can compute the contribution to the Chern number of the dual from our set of sections (hence to $\xi$) as we did in this section to the dual of the Nash bundle. As in the Nash bundle case, the contribution will be a sum depending on the polar varieties of $E$ relative to $F$ and their intersections. These polar varieties provide some  measure of the geometry of $E$ at its singular points on $X$.

 \end{document}